\documentclass[11pt]{amsart}
\usepackage{amssymb,amscd}
\usepackage[all]{xy}
\usepackage{geometry}
\usepackage{fancybox}
\usepackage{fancyhdr}
\usepackage{graphicx}
\DeclareGraphicsExtensions{.jpg,.png}  
\begin{document}
\newtheorem{thm}{Theorem}[section]
\newtheorem{lem}[thm]{Lemma}
\newtheorem{prop}[thm]{Proposition}
\newtheorem{cor}[thm]{Corollary}
\theoremstyle{definition}
\newtheorem{ex}[thm]{Example}
\newtheorem{rem}[thm]{Remark}
\newtheorem{prob}[thm]{Problem}
\newtheorem{thmA}{Theorem}
\renewcommand{\thethmA}{}
\newtheorem{defi}[thm]{Definition}
\renewcommand{\thedefi}{}
\input amssym.def
\long\def\alert#1{\smallskip{\hskip\parindent\vrule%
\vbox{\advance\hsize-2\parindent\hrule\smallskip\parindent.4\parindent%
\narrower\noindent#1\smallskip\hrule}\vrule\hfill}\smallskip}
\def\ff{\frak}
\def\Spec{\mbox{\rm Spec}}
\def\type{\mbox{ type}}
\def\Hom{\mbox{ Hom}}
\def\rank{\mbox{ rank}}
\def\Ext{\mbox{ Ext}}
\def\Ker{\mbox{ Ker}}
\def\Max{\mbox{\rm Max}}
\def\End{\mbox{\rm End}}
\def\l{\langle\:}
\def\r{\:\rangle}
\def\Rad{\mbox{\rm Rad}}
\def\Zar{\mbox{\rm Zar}}
\def\Supp{\mbox{\rm Supp}}
\def\Rep{\mbox{\rm Rep}}
\def\cal{\mathcal}
\title[n-fold filters in residuated lattices]{n-fold filters in residuated lattices}
\thanks{\today}
\author{Albert Kadji, Celestin Lele, Marcel Tonga}
\address{Department of Mathematics, University of Yde 1 } \email{kadjialbert@yahoo.fr}
\address{Department of Mathematics, University of Dschang }\email{celestinlele@yahoo.com}
\address{Department of Mathematics, University of Yde 1 }\email{tongamarcel@yahoo.fr}
\begin{abstract}Residuated lattices play an important role in the
study of fuzzy logic based of t-norm. In this paper, we introduced
the notions of n-fold implicative filters, n-fold positive
implicative filters, n-fold boolean  filters, n-fold fantastic
filters, n-fold normal  filters and  n-fold obstinate  filters in
residuated lattices  and study the relations among them.\\
This generalized the similar existing results  in BL-algebra with the connection of the
 work of Kerre and all in \cite{B11},  Kondo and all in \cite{B6}, \cite{B8} and Motamed and all in \cite{B7}.\\
 At the end of this paper, we draw two diagrams; the
 first one describe the relations between some type of n-fold
 filters in residuated lattices and the second one describe the relations
  between some type of n-fold residuated lattices.
 \vspace{0.20in}\\
{\noindent} Key words: residuated lattices, filters, n-fold filters,
n-fold  residuated lattices.
\end{abstract}
 \maketitle
\section{Introduction}
Since H\'{a}jek introduced his Basic Fuzzy logics, (BL-logics) in
short in 1998 \cite{B1}, as logics of continuous t-norms, a
multitude research papers related to algebraic counterparts of
BL-logics, has been published. In \cite{B2}, \cite{B3},\cite{B7}
and \cite{B10}  the authors defined the notion of n-fold
implicative filters, n-fold positive implicative filters, n-fold
boolean filters, n-fold fantastic filters, n-fold obstinate
filters, n-fold normal filters in
BL-algebras and studied the relation among many type of n-fold filters in BL-algebra.\\
The aim of this paper is to extend this research
to  residuated lattices with the connection of the results obtaining in \cite{B11}, \cite{B8}, \cite{B6}. \\
\section{Preliminaries}
A \textit{residuated lattice} is a nonempty set $L$ with four
binary operations
 $\wedge, \vee, \otimes, \rightarrow$, and two constants $0,1$ satisfying:\\
L-1: $\mathbb{L}(L):=(L,\wedge, \vee, 0, 1)$ is a bounded lattice;\\
L-2: $(L,\otimes, 1)$ is a commutative monoid;\\
L-3: $x\otimes y\leq z$ iff $x\leq y\rightarrow z$ (Residuation);\\
A \textit{MTL-algebra} is a residuated lattice  $L$ which satisfies the following  condition: \\
L-4: $(x\rightarrow y)\vee (y\rightarrow x)=1$ (Prelinearity);\\
A \textit{BL-algebra} is a MTL-algebra  $L$ which satisfies the following condition: \\
L-5: $x\wedge y=x\otimes (x\rightarrow y)$ (Divisibility).\\
A \textit{MV-algebra} is a BL-algebra  $L$ which satisfies the following  condition: \\
L-6: $ \overline{\overline{x}}=x $ where $\overline{x}:=x\rightarrow
0$.

In this work, unless mentioned otherwise, $(L,\wedge, \vee, \otimes,
\rightarrow, 0, 1)$  will be a residuated lattice, which
will often be referred by its support set $L$.\\
\begin{prop} \label{pro}
\cite{B4},\cite{B5},\cite{B6},\cite{B8}For all $x,y,z \in L$
\begin{align}
& x\leq y \; \text{iff}\; x\rightarrow y=1; x\otimes y \leq x\wedge y;\\
& x\rightarrow(y\rightarrow z)=(x\otimes y)\rightarrow z;\\
& x\rightarrow (y\rightarrow z)=y\rightarrow (x\rightarrow z);\\
& \text{If}\; x\leq y, \; \text{then}\; y\rightarrow z\leq x\rightarrow z\; \text{and}\; z\rightarrow x\leq z\rightarrow y;\\
& x\leq y\rightarrow (x\otimes y);  x\otimes (x\rightarrow y)\leq y;\\
& 1\rightarrow x=x; x\rightarrow x=1; x\rightarrow 1=1; x\leq y\rightarrow x, x\leq \bar{\bar{x}}, \bar{\bar{\bar{x}}}=\bar{x};\\
& x\otimes \bar{x}= 0; x\otimes y=0\; \text{iff}\;   x\leq \bar{y};\\
& x\leq y\; \text{implies}\;  x\otimes z \leq y\otimes z,  z \rightarrow x \leq  z \rightarrow y, y \rightarrow z \leq  x \rightarrow z, \bar{y} \leq \bar{x} ;\\
& \overline{x\otimes y}  =  x \rightarrow \bar{y};\\
& x\vee y= 1 \; \text{implies}\; x\otimes y = x\wedge y \; \text{and}\ x^{n}\vee y^{n}= 1\; \text{for every }\; n\geq 1;\\
& x\otimes (y\vee z)= (x\otimes y)\vee (x\otimes z);\\
& (x\vee y)\rightarrow z= (x\rightarrow z)\wedge(y\rightarrow z); (x\rightarrow z)\vee(y\rightarrow z)\leq (x\wedge y)\rightarrow z ;\\
& (x\vee y)\otimes (x\vee z)\leq x\vee (y\otimes z),\; \text{hence}\; ( x\vee y)^{mn}\leq x^{n}\vee y^{m}   ;\\
& x\vee y\leq ((x\rightarrow y)\rightarrow y)\wedge ((y\rightarrow x)\rightarrow x);\\
& x\rightarrow y\leq (y\rightarrow z)\rightarrow (x\rightarrow z);\\
& y\rightarrow x\leq (z\rightarrow y)\rightarrow (z\rightarrow x);\\
&((x\rightarrow y)\rightarrow y)\rightarrow y=x\rightarrow y .
\end{align}
\end{prop}
Besides equations (1)-(17), we will use the following results.\\

 Fact 1\\ A nonempty subset $F$ of a residuated lattice $(L,\wedge, \vee, \otimes,
\rightarrow, 0, 1)$ is called a \textit{filter} if it satisfies:
(F1): For every $x,y\in F$, $x\otimes y\in F$;\\
(F2): For every $x,y\in L$, if $x\leq y$ and $x\in F$, then $y\in F$.\\
A deductive system of a residuated lattices $L$ is a subset $F$
containing $1$ such that for all $x,y\in L$;
$  x\rightarrow  y \in F  \ \ \text{and} \ \ x\in F \ \ \text{imply}\ \ y\in F.$\\
It is known that in a residuated lattices, filters and deductive systems coincide \cite{B4}. \\
Fact 2\\ The following Examples will be use as a residuated lattices
which are not
  BL-algebra.

\begin{ex}\cite{B9} \label{expob1}
Let $ L=\{ 0, a, b, c, d, 1 \} $ be a lattice such that $0 < a <c
$, $0< b < c < d< 1 $, $a$ and $b$ are incomparable.
  Define the operations $\otimes $ and $\rightarrow $ by the two tables. Then
  $L$ is a residuated lattice which is not a
  BL-algebra since $(a\longrightarrow b)\vee (b\longrightarrow a)= c\neq 1$.
\begin{center}
\begin{tabular}{|c|c|c|c|c|c|c|}
\hline
 $\otimes$ & $0$ & $a$ & $b$ & $c$& $d$ & $1$ \\
  \hline
  $0$ & $0$ & $0$ & $0$ & $0$& $0$ & $0$  \\
  $a$ & $0$ & $a$ & $0$ & $a$& $a$ & $a$  \\
  $b$ & $0$ & $0$ & $b$ & $b$& $b$ & $b$  \\
  $c$ & $0$ & $a$ & $b$ & $c$& $c$ & $c$   \\
  $d$ & $0$ & $a$ & $b$ & $c$& $c$ & $d$   \\
  $1$ & $0$ & $a$ & $b$ & $c$& $d$ & $1$   \\
    \hline
\end{tabular}
\hspace{2 cm}
\begin{tabular}{|c|c|c|c|c|c|c|}
  \hline
 $\longrightarrow$ & $0$ & $a$ & $b$ & $c$& $d$ & $1$ \\
  \hline
  $0$ & $1$ & $1$ & $1$ & $1$& $1$ & $1$  \\
  $a$ & $b$ & $1$ & $b$ & $1$& $1$ & $1$  \\
  $b$ & $a$ & $a$ & $1$ & $1$& $1$ & $1$  \\
  $c$ & $0$ & $a$ & $b$ & $1$& $1$ & $1$  \\
  $d$ & $0$ & $a$ & $b$ & $d$& $1$ & $1$  \\
  $1$ & $0$ & $a$ & $b$ & $c$& $d$ & $1$  \\
    \hline
\end{tabular}
\end{center}
 $F= \{1,b,c,d\}$; $F_{1}= \{1,a,c,d\}$; $F_{2}= \{1,c,d\}$  are proper filters  of $L$.
\end{ex}
\begin{ex}\cite{B9} \label{expob2}
Let $ L=\{ 0, a, b, c, d, 1 \} $ be a lattice such that
$0<a,b,d,c<1 $,  $a,b,c,d$  are pairwise  incomparable.
  Define the operations $\otimes $ and $\rightarrow $ by the two tables. Then
  $L$ is a residuated lattice which is not a
  BL-algebra since $a\otimes(a\longrightarrow b)= b\neq 0= a\wedge b$.
\begin{center}
\begin{tabular}{|c|c|c|c|c|c|c|}
\hline
 $\otimes$ & $0$ & $a$ & $b$ & $c$& $d$ & $1$ \\
  \hline
  $0$ & $0$ & $0$ & $0$ & $0$& $0$ & $0$  \\
  $a$ & $0$ & $a$ & $b$ & $d$& $d$ & $a$  \\
  $b$ & $0$ & $b$ & $b$ & $0$& $0$ & $b$  \\
  $c$ & $0$ & $d$ & $0$ & $d$& $d$ & $c$   \\
  $d$ & $0$ & $d$ & $0$ & $d$& $d$ & $d$   \\
  $1$ & $0$ & $a$ & $b$ & $c$& $d$ & $1$   \\
    \hline
\end{tabular}
\hspace{2 cm}
\begin{tabular}{|c|c|c|c|c|c|c|}
  \hline
 $\longrightarrow$ & $0$ & $a$ & $b$ & $c$& $d$ & $1$ \\
  \hline
  $0$ & $1$ & $1$ & $1$ & $1$& $1$ & $1$  \\
  $a$ & $a$ & $1$ & $b$ & $c$& $c$ & $1$  \\
  $b$ & $c$ & $a$ & $1$ & $c$& $c$ & $1$  \\
  $c$ & $b$ & $a$ & $b$ & $1$& $a$ & $1$  \\
  $d$ & $b$ & $a$ & $b$ & $a$& $1$ & $1$  \\
  $1$ & $0$ & $a$ & $b$ & $c$& $d$ & $1$  \\
    \hline
\end{tabular}
\end{center}
 $F= \{1,c,d\}$ is a proper filter  of $L$.
\end{ex}
 \begin{ex}\cite{B4} \label{exp}
Let $ L=\{ 0, a, b, c, d, 1 \} $ be a lattice such that $0 < a<c<d
<1 $, $0< b < c < d< 1 $, $a$ and $b$ are incomparable.
  Define the operations $\otimes $ and $\rightarrow $ by the two tables. Then
  $(L, \wedge, \vee, \otimes, \rightarrow, 0, 1)$ is a residuated lattices which is
  not a BL-algebra since $(a\longrightarrow b)\vee (b\longrightarrow a)= c\vee c=c\neq
  1$.
\begin{center}
\begin{tabular}{|c|c|c|c|c|c|c|}
\hline
 $\otimes$ & $0$ & $a$ & $b$ & $c$& $d$ & $1$ \\
  \hline
  $0$ & $0$ & $0$ & $0$ & $0$& $0$ & $0$  \\
  $a$ & $0$ & $0$ & $0$ & $0$& $a$ & $a$  \\
  $b$ & $0$ & $0$ & $0$ & $0$& $b$ & $b$  \\
  $c$ & $0$ & $0$ & $0$ & $0$& $c$ & $c$   \\
  $d$ & $0$ & $a$ & $b$ & $c$& $d$ & $d$   \\
  $1$ & $0$ & $a$ & $b$ & $c$& $d$ & $1$   \\
    \hline
\end{tabular}
\hspace{2 cm}
\begin{tabular}{|c|c|c|c|c|c|c|}
  \hline
 $\longrightarrow$ & $0$ & $a$ & $b$ & $c$& $d$ & $1$ \\
  \hline
  $0$ & $1$ & $1$ & $1$ & $1$& $1$ & $1$  \\
  $a$ & $c$ & $1$ & $c$ & $1$& $1$ & $1$  \\
  $b$ & $c$ & $c$ & $1$ & $1$& $1$ & $1$  \\
  $c$ & $c$ & $c$ & $c$ & $1$& $1$ & $1$  \\
  $d$ & $0$ & $a$ & $b$ & $c$& $1$ & $1$  \\
  $1$ & $0$ & $a$ & $b$ & $c$& $d$ & $1$  \\
    \hline
\end{tabular}
\end{center}
 $F= \{1,d\}$ is a proper filter of $L$.
\end{ex}

 \begin{ex}\cite{B5} \label{expp}
Let $ L=\{ 0, a, b, c,  1 \} $ be a lattice such that $0 <c<a, b<1$,
 $a, b$ are incomparable.
  Define the operations $\otimes $ and $\rightarrow $ by the two tables. Then
  $(L, \wedge, \vee, \otimes, \rightarrow, 0, 1)$ is a residuated lattice which is not
  BL-algebra since $a\otimes(a\longrightarrow b)=c\neq 0=a\wedge b$.
\begin{center}
\begin{tabular}{|c|c|c|c|c|c|}
\hline
 $\otimes$ & $0$ & $c$ & $a$ & $b$&  $1$ \\
  \hline
  $0$ & $0$ & $0$ & $0$ & $0$& $0$  \\
  $c$ & $0$ & $c$ & $c$ & $c$& $c$   \\
  $a$ & $0$ & $c$ & $a$ & $c$& $a$   \\
  $b$ & $0$ & $c$ & $c$ & $b$& $b$    \\
  $1$ & $0$ & $c$ & $a$ & $b$& $1$    \\

    \hline
\end{tabular}
\hspace{2 cm}
\begin{tabular}{|c|c|c|c|c|c|}
  \hline
 $\longrightarrow$ & $0$ & $c$ & $a$ & $b$&  $1$ \\
  \hline
  $0$ & $1$ & $1$ & $1$ & $1$& $1$  \\
  $c$ & $0$ & $1$ & $1$ & $1$& $1$  \\
  $a$ & $0$ & $b$ & $1$ & $b$& $1$ \\
  $b$ & $0$ & $a$ & $a$ & $1$& $1$ \\
  $1$ & $0$ & $c$ & $a$ & $b$& $1$ \\

    \hline
\end{tabular}
\end{center}
 $F_{1}= \{1,a\}, F_{2}= \{1,b\},  F_{3}= \{1,a,b,c\}$ are  proper filters of $L$.
\end{ex}
\begin{defi}\cite{B8}
A residuated lattice   $L$ is said to be locally finite if for
every $x\neq 1$, there exists an integer $n\geq 1$ such that
$x^{n}:=\underbrace{x\otimes x\cdots \otimes x}_{n\; times} =0$.
\end{defi}
\begin{defi}\cite{B6}\
Let $F$ be a filter of a residuated lattice $L$. For  $x, y \in
L$, a relation $\equiv_{F}$ on $L$, define  by $x \equiv_{F} y
\Longleftrightarrow (x\longrightarrow y, y\longrightarrow x) \in
F$, is a congruence on $L$ and a quotient structure  $L/F$ is also
a  residuated lattice where : $x/F\wedge y/F= (x\wedge y)/F$;
$x/F\vee/F= (x\vee y)/F$;  $x/F\otimes y/F= (x\otimes y)/F$;
$x/F\longrightarrow y/F= (x\longrightarrow y)/F$.
\end{defi}
\begin{defi} \cite{B11} \cite{B8} \cite{B4}  A Proper filter  $F$ is said to be:
\begin{itemize}
 \item[(i)] prime   if it satisfies the
following condition:\\ For all $x,y\in L$, $x\longrightarrow y \in
F$ or  $ y \longrightarrow x\in F$.
\item[(ii)] prime of the second kind  if it
satisfies the following condition:\\ For all $x,y\in L$, $x\vee y
\in F$ implies  $x \in F$ or  $ y \in F$.
\item[(iii)] prime of the third kind  if it
satisfies the following condition:\\ For all $x,y\in L$,
$(x\longrightarrow y) \vee (y \longrightarrow x)\in F$.
\item[(iv)]boolean
if it satisfies the following condition:\\ For all $x\in L$, $x
\vee\overline{x} \in F$.
\item[(v)]boolean filter in the second kind
if it satisfies the following condition:\\ For all $x\in L$, $x \in
F$ or  $\overline{x} \in F$.
\end{itemize}
\end{defi}
\begin{rem} \cite{B11}\cite{B8} \cite{B5}
\begin{itemize}
 \item[(i)]Prime filters are prime filters  in the second kind. The converse is
true if   $L$  is a MTL-algebra.
\item[(ii)]Prime filters are prime filters in the third  kind. The converse
is  true if $L$  is a MTL-algebra.
\item[(iii)]Boolean filters in the second kind are boolean filters.
\item[(iv)]   Maximal filters
are prime filters in the second kind.
\item[(v)]If  $L$  is a MTL-algebra, then   maximal filters
are prime filters.
\end{itemize}
\end{rem}
We have the following results.
\begin{prop} \cite{B5} \label{max} For any filter  $F$ of a
residuated lattices  $L$, the following conditions are equivalent:
\begin{itemize}
 \item[(i)]  $F$ is a maximal filter of $L$.
\item[(ii)] For any  $x \in L$, $x \notin F$  if and only if  $\overline{x^{n}} \in F$ for
some $n \geq  1$
\item[(iii)]For any  $x \notin F$, there is $f\in F$ and $n \geq  1$
such that $f\otimes x^{n} =0$
\end{itemize}
\end{prop}
Follows from Prop.\ref{max}, we have the following lemma:
\begin{lem}
$F$ is a maximal filter of $L$ if and only if  $L/F$ is a locally
finite residuated lattice.
\end{lem}

Now, unless mentioned otherwise,  $n \geq 1$ will be an integer and
$ F\subseteq L$.
\section{SEMI MAXIMAL FILTER IN RESIDUATED LATTICES}
\begin{defi}\cite{B12}
Let $F$ be a proper filter of $L$. The intersection of  all maximal
filters of $L$ which contain $F$ is called the radical of $F$ and it
is denoted by Rad($F$).
\end{defi}
\begin{defi}
A proper filter  $F$ of $L$ is said to be a semi maximal filter of
$L$ if Rad($F$)= $F$.
\end{defi}
The following example shows that the notion of semi maximal filters
in  residuated lattices exist and semi maximal filter may not be
maximal filter.
\begin{ex}
Let $L$ be a residuated lattice from Example \ref{expob1}. It is
easy to check that Rad($\{1,c,d\}$)= $\{1,c,d\}$. Hence $\{1,c,d\}$
is a semi maximal filter of $L$.\\ But  $\{1,c,d\}\subseteq
\{1,a,c,d\}$ and $\{1,c,a,d\}$ is a filter of $L$, hence $\{1,c,d\}$
is a semi maximal filter which is not a  maximal filter of $L$.
\end{ex}
\begin{rem}
It is clear that  maximal filters are semi maximal filters.
\end{rem}
\section{N-FOLD IMPLICATIVE FILTER IN RESIDUATED LATTICES}
\begin{defi}
An n-fold implicative residuated lattice $L$  is a residuated
lattices which satifies the following condition: \\
 $x^{n+1}= x^{n} $ for all  $x,y \in  L $.
\end{defi}
The following examples shows that n-fold implicative residuated
lattices exist and that residuated lattice is not in general n-fold
implicative residuated lattice.
\begin{ex}
Let $L$ be a residuated lattice from Example \ref{exp}. We have:
\begin{itemize}
\item[(i)]  $a^{1+1}= 0\neq  a^{1}$ so
$L$ is not an 1-fold implicative residuated lattice.
\item[(ii)] For all $n \geq  2$,  $x^{n+1}= x^{n}$
 for all  $x,y \in  L$. So $L$ is an n-fold implicative residuated lattice for all $n
\geq  2$.
\end{itemize}
\end{ex}
\begin{defi}\label{defip2}
$F$  is an n-fold implicative filter  if it satisfies the
following conditions:
\begin{itemize}
\item[(i)]$1\in F$
\item[(ii)]For all  $x, y, z \in L$, if $x^{n}\longrightarrow (y\longrightarrow z)\in F$
 and $x^{n}\longrightarrow y\in
F$, then $x^{n}\longrightarrow z\in F$.
\end{itemize}
In particular 1-fold implicative filters are  implicative
filters.\cite{B6}
\end{defi}
\begin{ex}\label{ex1}Let $n\geq  1$ and
 $L$ be a residuated lattice from Example \ref{expp}.  Simple
computations proves that  $F_{1}= \{1,a\}, F_{2}= \{1,b\}, F_{3}=
\{1,a,b,c\}$ are n-fold implicative filters.
\end{ex}

The following lemma gives a characterization of n-fold implicative
filters.
\begin{lem}\label{lemob} Let $a\in L$.
Let   $F$ be a filter of $L$. Then $L_{a}=\{b\in L:
a^{n}\longrightarrow b\in F  \}$ is a filter of $L$ if and only if
$F$ is an n-fold implicative filter of $L$.
\end{lem}
\begin{proof}
Let   $F$ be an n-fold implicative filter of $L$.  Since
$a^{n}\longrightarrow 1=1 \in F$, we have $ 1 \in L_{a}$.  Let
$x,y \in L$ be such that $x, x\longrightarrow y \in L_{a}$, then
$a^{n}\longrightarrow x \in F$ and  $a^{n}\longrightarrow
(x\longrightarrow y)\in F$. Since  $F$ is an n-fold implicative
filter of $L$, by Definition \ref {defip2}, $a^{n}\longrightarrow
y \in F$, hence $ y \in L_{a}$. Therefore $L_{a}$ is a filter of
$L$.\\ Conversely suppose that $L_{a}$ is a filter of $L$ for all
$a\in L$. Let $x, y, z \in L$ be such that $x^{n}\longrightarrow
(y\longrightarrow z)\in F$
 and $x^{n}\longrightarrow y\in
F$.  We have  $y, y\longrightarrow z \in L_{x}$, by the hypothesis
$L_{x}$ is a filter of $L$, so $ z \in L_{x}$ and hence
$x^{n}\longrightarrow z\in F$.
\end{proof}
The following proposition gives another characterization of n-fold
implicative filters in residuated lattices.
\begin{prop}\label{ch111} Let $F$ be a filter of $L$. Then for all  $x \in L$, the following conditions are equivalent:
\begin{itemize}
\item[(i)]$F$ is an n-fold implicative filter of $L$.
\item[(ii)]   $x^{n}\longrightarrow x^{2n}\in F$.
\end{itemize}
\end{prop}
\begin{proof}
 $(i)\longrightarrow (ii)$:\ Let  $x \in L$, by Prop. \ref{pro}
  we have: $x^{n}\longrightarrow ( x^{n}\longrightarrow x^{2n})= x^{2n}\longrightarrow x^{2n}=1\in F$
and $x^{n}\longrightarrow  x^{n}=1\in F$. Since $F$ is an n-fold
implicative filter of $L$, we get $x^{n}\longrightarrow x^{2n}\in
F$.\\
 $(ii)\longrightarrow (i)$:\  Let $x, y, z \in
L$ be such that  $x^{n}\longrightarrow (y\longrightarrow z)\in F$
 and $x^{n}\longrightarrow y\in F$. By Prop. \ref{pro} we  have the
 following:
\begin{itemize}
\item[(1)]$x^{n}\otimes [x^{n}\longrightarrow (y\longrightarrow z)]\leq y\longrightarrow
z$.
\item[(2)]$x^{n}\otimes (x^{n}\longrightarrow y)\leq y $.
\item[ (3)]By (1) and (2) we have : $[x^{n}\otimes [x^{n}\longrightarrow (y\longrightarrow
z)]]\otimes [x^{n}\otimes (x^{n}\longrightarrow y)]\leq y\otimes
(y\longrightarrow z)\leq z$.
\item[ (4)] By (3) we have : $([x^{n}\longrightarrow (y\longrightarrow
z)]\otimes (x^{n}\longrightarrow y))\otimes x^{2n}\leq z$.
\item[ (5)] By (4), we have :$([x^{n}\longrightarrow (y\longrightarrow
z)]\otimes (x^{n}\longrightarrow y))\leq x^{2n}\longrightarrow z$.
\item[ (6)]Since $x^{n}\longrightarrow (y\longrightarrow z)\in F$
 and $x^{n}\longrightarrow y\in F$, by the fact that $F$ is a
 filter, we get $[x^{n}\longrightarrow (y\longrightarrow z)]\otimes(x^{n}\longrightarrow y)\in F$
\item[ (7)]By (5),(6) and  the fact that $F$ is a
 filter,  we get  $x^{2n}\longrightarrow z\in F$.
\item[ (8)]  $x^{n}\longrightarrow x^{2n}\leq (x^{2n}\longrightarrow z)\longrightarrow
 (x^{n}\longrightarrow z) $
\item[ (9)]By (7), (8)  and  the fact that  $x^{n}\longrightarrow x^{2n} \in
F$, we obtain $x^{n}\longrightarrow z\in F$.
\end{itemize}
Hence $F$ is an n-fold implicative filter of $L$.
\end{proof}
\begin{prop}\label{ch112} Let $F$ be a filter of $L$. Then for all  $x, y \in
L$, the following conditions are equivalent:
\begin{itemize}

\item[(i)]   $x^{n}\longrightarrow x^{2n}\in F$.
\item[(ii)] If  $x^{n+1}\longrightarrow y\in F$, then
$x^{n}\longrightarrow y\in F$.

\end{itemize}
\end{prop}
\begin{proof}
$(i)\longrightarrow (ii)$:\ Since $(i)$ holds,  by Prop. \ref{ch111}
$F$ is an n-fold implicative filter of $L$.  On the other
 hand by  Prop. \ref{pro} we have :
  $x^{n+1}\longrightarrow y =x^{n}\longrightarrow (x\longrightarrow
y) \in F$ and $x^{n}\longrightarrow x= 1\in F$, by the fact that $F$
is an n-fold implicative filter of $L$ we get $x^{n}\longrightarrow y\in F$.\\
 $(ii)\longrightarrow (i)$:\ We have :\\
 $x^{n+1}\longrightarrow(x^{n-1}\longrightarrow x^{2n})= x^{2n}\longrightarrow x^{2n}=1\in F$.
From this and the fact that (ii) holds, we  also have : \\
$x^{n}\longrightarrow(x^{n-1}\longrightarrow x^{2n})\in F$.\\ But
$x^{n+1}\longrightarrow(x^{n-2}\longrightarrow
x^{2n})=x^{n}\longrightarrow(x^{n-1}\longrightarrow
x^{2n})\in F$.\\
From this and the fact that (ii) holds, we  also have : \\
$x^{n}\longrightarrow(x^{n-2}\longrightarrow x^{2n})\in F$.\\
By repeating the process n times, we get
$x^{n}\longrightarrow(x^{n-n}\longrightarrow
x^{2n})=x^{n}\longrightarrow x^{2n}\in F$.
\end{proof}
\begin{prop}\label{ch113} Let $F$ be a filter of $L$. Then for all  $x, y, z \in
L$, the following conditions are equivalent:
\begin{itemize}
\item[(i)]$x^{n}\longrightarrow x^{2n}\in F$.
\item[(ii)] If  $x^{n}\longrightarrow ( y
\longrightarrow z)\in F$, then  $(x^{n}\longrightarrow
y)\longrightarrow ( x^{n}\longrightarrow z)\in F$.
\end{itemize}
\end{prop}
\begin{proof}
$(i)\longrightarrow (ii)$:\ Assume that $x^{n}\longrightarrow ( y
\longrightarrow z)\in F$. By Prop. \ref{pro} we  have the
 following equations:
\begin{itemize}
\item[(1)] $ y\longrightarrow z\leq (x^{n}\longrightarrow y)\longrightarrow (x^{n}\longrightarrow z)$.
\item[(2)] $ x^{n}\longrightarrow(y\longrightarrow z)\leq x^{n}\longrightarrow[(x^{n}\longrightarrow y)\longrightarrow (x^{n}\longrightarrow z)]$.
\item[ (3)]$x^{n}\longrightarrow[(x^{n}\longrightarrow y)\longrightarrow (x^{n}\longrightarrow z)]=
x^{n}\longrightarrow[x^{n}\longrightarrow ((x^{n}\longrightarrow
y)\longrightarrow z)]=x^{2n}\longrightarrow[ (x^{n}\longrightarrow
y)\longrightarrow z]$.
\item[ (4)] By (2) and (3) we have : $ x^{n}\longrightarrow(y\longrightarrow
z)\leq x^{2n}\longrightarrow[ (x^{n}\longrightarrow
y)\longrightarrow z]$.
\item[ (5)]Since $F$ is a filter, by (4) and the fact that  $x^{n}\longrightarrow ( y
\longrightarrow z)\in F$, we have : $x^{2n}\longrightarrow[
(x^{n}\longrightarrow y)\longrightarrow z]\in F$
\item[ (6)]$x^{2n}\longrightarrow[
(x^{n}\longrightarrow y)\longrightarrow z]\leq(x^{n}\longrightarrow
x^{2n})\longrightarrow(x^{n}\longrightarrow[ (x^{n}\longrightarrow
y)\longrightarrow z])=(x^{n}\longrightarrow x^{2n})\longrightarrow
((x^{n}\longrightarrow y)\longrightarrow[ (x^{n}\longrightarrow
z])$.
\item[ (7)]Since $F$ is a filter, by (6) and the fact that  $x^{2n}\longrightarrow[
(x^{n}\longrightarrow y)\longrightarrow z]\in F$,  we have :
$(x^{n}\longrightarrow x^{2n})\longrightarrow ((x^{n}\longrightarrow
y)\longrightarrow[ (x^{n}\longrightarrow z])\in F$.
\item[(8)]Since $F$ is a filter, by (7) and the fact that
 $x^{n}\longrightarrow x^{2n}\in F $, we obtain $(x^{n}\longrightarrow
y)\longrightarrow ( x^{n}\longrightarrow z)\in F$.\\
\end{itemize}
$(ii)\longrightarrow (i)$:\ Since
$x^{n}\longrightarrow(x^{n}\longrightarrow x^{2n})=
x^{2n}\longrightarrow x^{2n}=1\in F$, by (ii) we have :
$(x^{n}\longrightarrow x^{n})\longrightarrow(x^{n}\longrightarrow
x^{2n})\in F$, hence $x^{n}\longrightarrow x^{2n}\in F$.
\end{proof}

 By Prop. \ref{ch111},   Prop. \ref{ch112} and   Prop. \ref{ch113},
 we have the following result:
\begin{prop}\label{ch1} Let $F$ be a filter of $L$. Then for all  $x, y, z \in
L$, the following conditions are equivalent:
\begin{itemize}
\item[(i)]$F$ is an n-fold implicative filter of $L$.
\item[(ii)]   $x^{n}\longrightarrow x^{2n}\in F$.
\item[(iii)] If  $x^{n+1}\longrightarrow y\in F$, then
$x^{n}\longrightarrow y\in F$.
\item[(iv)]If  $x^{n}\longrightarrow ( y \longrightarrow z)\in
F$, then  $(x^{n}\longrightarrow y)\longrightarrow
(x^{n}\longrightarrow z)\in F$.
\end{itemize}
\end{prop}
\begin{prop}\label{ch2}
 If a filter $F$ is an n-fold implicative filter,
 then $F$ is an (n+1)-fold implicative filter.
\end{prop}
\begin{proof}Let  $F$ be  a filter.
Assume that $F$ is an n-fold implicative filter. Let  $x, y \in L$
such that  $x^{n+2}\longrightarrow y\in F$, by Prop.  \ref{pro},
$x^{n+1}\longrightarrow(x\longrightarrow y)=x^{n+2}\longrightarrow
y\in F$. Since  $F$ is an n-fold implicative filter, apply Prop.
\ref{ch1}(iii),\\ we obtain $x^{n}\longrightarrow(x\longrightarrow
y)\in F$. Hence $x^{n+1}\longrightarrow y\in F$ and by Prop.
\ref{ch1},  $F$ is an (n+1)-fold implicative filter.
 \end{proof}
By the following example, we show that the converse of Prop.
\ref{ch2}  is not true in general.
\begin{ex}\label{exim3} \cite{B13}

Let $ L=\{ 0, a, b,   1 \} $ be a lattice such that $0 <a<b<1$.
  Define the operations $\otimes $ and $\rightarrow $ by the two tables. Then
  $(L, \wedge, \vee, \otimes, \rightarrow, 0, 1)$ is a residuated
  lattice.
\begin{center}
\begin{tabular}{|c|c|c|c|c|}
\hline
 $\otimes$ & $0$  & $a$ & $b$&  $1$ \\
  \hline
  $0$ & $0$ & $0$ & $0$ & $0$ \\
  $a$ & $0$ & $0$ & $0$ & $a$  \\
  $b$ & $0$ & $0$ & $a$ & $b$   \\
  $1$ & $0$ & $a$ & $b$ & $1$    \\
  \hline
\end{tabular}
\hspace{2 cm}
\begin{tabular}{|c|c|c|c|c|}
  \hline
 $\longrightarrow$ & $0$  & $a$ & $b$&  $1$ \\
  \hline
  $0$ & $1$ & $1$ & $1$ & $1$  \\
  $a$ & $b$ & $1$ & $1$ & $1$  \\
  $b$ & $a$ & $b$ & $1$ & $1$ \\
  $1$ & $0$ & $a$ & $b$ & $1$\\
 \hline
\end{tabular}
\end{center}
 $\{1\}$ is an 3-fold implicative filter but $\{1\}$ is not an 2-fold implicative
 filter, since $b^{2}\longrightarrow a \in\{1\}$ but $b^{1}\longrightarrow a = b \notin\{1\}$
\end{ex}
\begin{prop}\label{ch11}
n-fold implicative filters are filters.
\end{prop}
\begin{proof}
Suppose that $F$ is an n-fold implicative filter of $L$. Let
 $z, y \in L$ such that  $y, y\longrightarrow z\in F$.
  We have  $1^{n}\longrightarrow y, 1^{n}\longrightarrow ( y\longrightarrow z)\in
  F$, this implies $z=1^{n}\longrightarrow z\in F$. Hence $F$ is a  deductive system
 of  $L$ and the thesis follows from the fact1.
\end{proof}

By the following example, we show that the converse of Prop.
\ref{ch11} is not true in general.

\begin{ex}
Let $L$ be a residuated lattice from Example \ref{exim3}. $\{1\}$ is
a filter but $\{1\}$ is not an 2-fold implicative
 filter, since $b^{2}\longrightarrow a \in\{1\}$ but $b^{1}\longrightarrow a = b
 \notin\{1\}$.
\end{ex}
Using Prop.\ref{ch1}, it is easy to show the following results:

\begin{cor}\label{ch3}
  If $L$ is an n-fold implicative residuated lattice then,  the
concepts of n-fold implicative filters and filters coincide.
\end{cor}
\begin{thm}\label{ch4}
Let $F_{1}$ and $F_{2}$ two filters of $L$  such that
$F_{1}\subseteq F_{2}$. If   $F_{1}$ is an n-fold implicative
filter, then  $F_{2}$ is an n-fold implicative filter.
\end{thm}

The following theorem gives the relation between  n-fold implicative
residuated lattice and n-fold implicative filter.

\begin{prop}\label{ch5} Let $F$ be a filter of $L$. The following conditions are equivalent:
\begin{itemize}
\item[(i)]$L$ is an n-fold implicative residuated lattice.
\item[(ii)]   Every filter of $L$ is  an n-fold implicative filter of
$L$.
\item[(iii)] \{1\}   is  an n-fold implicative filter of
$L$.
\item[(iv)] $x^{n}= x^{2n}$ for all
$x\in L$.
\end{itemize}
\end{prop}

\begin{proof}
$(i)\longrightarrow (ii)$ :  \ follows from Corollary. \ref{ch3} \\
$(ii)\longrightarrow (iii)$ :  \ follows from the fact that \{1\}
is  a filter of $L$.  \\
$(iii)\longrightarrow (iv)$ :  \  Assume that \{1\}   is  an
n-fold implicative filter of $L$. From Prop \ref{ch1}, we have
$x^{n}\longrightarrow x^{2n} = 1 $ for all $x\in L$. So $x^{n}\leq
x^{2n} $ for all $x\in L$. Since $x^{2n}\leq x^{n} $ for all $x\in
L$, we obtain $x^{n}= x^{2n}$ for all $x\in L$.\\
$(iv)\longrightarrow (i)$ :  \ If $x^{n}= x^{2n}$ for all $x\in
L$, we have $x^{n}\longrightarrow x^{2n} = 1 \in \{1\}$ for all
$x\in L$, by Prop. \ref{ch1},  $\{1\}$   is  an n-fold implicative
filter of $L$. Since $x^{n}\longrightarrow (x^{n}\longrightarrow
x^{n+1}) = 1 \in \{1\}$ and $x^{n}\longrightarrow x^{n} = 1 \in
\{1\}$, we get $x^{n}\longrightarrow x^{n+1} \in \{1\}$, that is
$x^{n+1}=x^{n}$ for all $x\in L$.
\end{proof}
\begin{cor}\label{ch6}
 A filter $F$ of  a residuated lattice $L$ is  an n-fold  implicative
 filter if and only if $L/F$  is an n-fold implicative residuated lattice.
\end{cor}
\begin{proof}
 Let  $F$ be a filter.\\ Suppose that $F$  is  an n-fold  implicative
 filter. By Prop. \ref{ch1}(ii),  we have $x^{n}\longrightarrow x^{2n}\in F$ for all
$x\in L$, that is $(x^{n}\longrightarrow x^{2n})/F= 1/F$ for all
$x\in L$. So $ (x/F)^{n}\longrightarrow (x/F)^{2n}=
(x^{n}/F)\longrightarrow (x^{2n}/F)= (x^{n}\longrightarrow
x^{2n})/F=1/F$ for all $x/F\in L/F$, by Prop. \ref{ch5}(iv), $L/F$
is an n-fold implicative residuated lattice.\\ Suppose conversely
that $L/F$  is an n-fold implicative residuated lattice.  By Prop.
\ref{ch5}(iv), we get $ (x/F)^{n}= (x/F)^{2n}$ for all $x/F\in
L/F$ or equivalently $ (x^{n}/F)= (x^{2n}/F)$ for all $x\in L$.
That is $(x^{n}\longrightarrow x^{2n})/F=1/F$ for all $x\in L$.
Hence $x^{n}\longrightarrow x^{2n}\in F$ for all $x\in L$, we
obtain the result by apply Prop. \ref{ch1}(ii).
\end{proof}
   By (3)\cite{B4} and Corollary  \ref{ch6}, we have the following
   result.
\begin{cor}\label{ch7}
 A filter $F$ of  a residuated lattice $L$ is  an 1-fold  implicative
 filter if and only if $L/F$  is a Heyting algebra. As a
 consequence, it is easy to observe that, a residuated lattice $L$   is a Heyting algebra if and only if
$\{1\}$ is  an 1-fold  implicative  filter of $L$ if and only if
$L$   is an 1-fold  implicative residuated lattice.
\end{cor}

\section{N-FOLD POSITIVE IMPLICATIVE FILTERS OF RESIDUATED
LATTICES}

\begin{defi}
$F$  is an n-fold positive implicative filter  if it satisfies the
following conditions:
\begin{itemize}
\item[(i)]$1\in F$
\item[(ii)]For all  $x, y, z \in L$,
 if $ x \longrightarrow ((y^{n}\longrightarrow z)\longrightarrow y)\in F$
 and $ x\in F$, then $ y\in F$.
\end{itemize}
In particular 1-fold positive implicative filters are positive
implicative filters.\cite{B6}
\end{defi}

\begin{ex}Let $n\geq  1$.
Let  $L$ be a residuated lattice from Example \ref{expp}. Simple
computations proves that $ F_{3}= \{1,a,b,c\}$ is an  n-fold
positive implicative filter.
\end{ex}

\begin{prop}\label{propo1}
Every  n-fold positive implicative filter is a filter.
\end{prop}
\begin{proof}
Let   $F$  be an n-fold positive implicative filter of $L$, it is
clear that $ 1\in F$.  Since for any $ y\in F$, $
y^{n}\longrightarrow 1=1$, by setting $ z=1$ in the definition of
n-fold positive implicative filter, we obtain the result.
\end{proof}

 The following Example shows that  filters  are not  n-fold
positive implicative filters in general.
\begin{ex}\label{ex2}
Let  $L$ be a residuated lattice from Example \ref{expp} and
$n\geq 1$. $F_{1}= \{1,a\}, F_{2}= \{1,b\}$ are filters   but not
n-fold positive implicative filters since $ 1\longrightarrow
((b^{n}\longrightarrow 0)\longrightarrow b)\in F_{1}$
 and $ 1\in F_{1}$, but $ b\notin F_{1}$;  $ 1\longrightarrow ((a^{n}\longrightarrow
0)\longrightarrow a)\in F_{2}$
 and $ 1\in F_{2}$, but $ a\notin F_{2}$.
\end{ex}

 The following proposition gives a characterization of  n-fold
positive implicative filter  for any $n\geq 1$ .

\begin{prop}\label{propo}
 The following conditions are
equivalent for any filter $F$ and any  $n\geq 1$ :
\begin{itemize}
\item[(i)]$F$ is an n-fold positive implicative filter
\item[(ii)]For all  $x, y \in L$, $(x^{n}\longrightarrow y)\longrightarrow x\in
F$ implies $x\in F$.
\item[ (iii)]For all  $x\in L$, $\overline{x^{n}}\longrightarrow x\in
F$ implies $x\in F$.
\end{itemize}
\end{prop}
\begin{proof}
$(i)\longrightarrow (ii)$  :    Suppose that  $F$ is n-fold positive
implicative filter of $L$ and $(x^{n}\longrightarrow
y)\longrightarrow x\in F$, since
$1\longrightarrow((x^{n}\longrightarrow y)\longrightarrow x)=
(x^{n}\longrightarrow y)\longrightarrow x\in F$ and  $ 1\in F$, we
apply the fact that $F$ is n-fold positive implicative filter of $L$
and obtain the result.\\
$(ii)\longrightarrow (iii)$ : We obtain the result by setting
$ y=0$ in the equation $(ii)$.\\
 $(iii)\longrightarrow (i)$  :  Suppose that $ x \longrightarrow
((y^{n}\longrightarrow z)\longrightarrow y)\in F$ and $ x\in F$,
from the fact that  $F$ is
 filter, we obtain  $ (y^{n}\longrightarrow
z)\longrightarrow y\in F$.\\ On the other hand, from Prop.
\ref{pro}(4), we have : $ (y^{n}\longrightarrow z)\longrightarrow
y\leq (y^{n}\longrightarrow 0)\longrightarrow y$, from the fact that
$F$ is filter, we obtain  $ (y^{n}\longrightarrow 0)\longrightarrow
y\in F$,  we apply the hypothesis and obtain $ y\in F$.
\end{proof}
\begin{cor}\label{lien11}
A proper filter  $F$ is an n-fold positive implicative filter if
an only if  for all $x\in L$, $x\vee\overline{x^{n}}\in F$.
\end{cor}
\begin{proof}
Assume that  for all $x\in L$,  $\overline{x^{n}}\longrightarrow x
\in F$ and $x\vee\overline{x^{n}}\in F$. By Prop. \ref{propo}, we
must show that $x\in F$.  Since by (14)Prop. \ref{pro},
$x\vee\overline{x^{n}} \leq (\overline{x^{n}}\longrightarrow
x)\longrightarrow x$, we have $(\overline{x^{n}}\longrightarrow
x)\longrightarrow x \in F$.
Using the fact that  $\overline{x^{n}}\longrightarrow x \in F$,  we have $x\in F$.\\
Conversely suppose that $F$ is an  n-fold positive implicative
filter. Let $x\in L$. Let  $t= x\vee\overline{x^{n}}$,  we must
show that $t\in F$. Since $x\leq t$, we have $x^{n}\leq t^{n}$ and
then $\overline{t^{n}}\leq \overline{x^{n}}\leq
\overline{x^{n}}\vee x=t$. Hence $\overline{t^{n}}\leq t$ or
equivalently $\overline{t^{n}} \longrightarrow t=1$. So
$\overline{t^{n}} \longrightarrow t \in F$. From this and the fact
 that $F$ is an n-fold positive implicative
filter, by Prop. \ref{propo}, we get that $t\in F$.
\end{proof}
\begin{defi}
$F$  is an n-fold boolean filter  if it satisfies the following
conditions:\\
 $x\vee \overline{x^{n}} \in F$ for all $x\in L$.
In particular 1-fold boolean filters are boolean
filters.\cite{B4}.
\end{defi}

The extension theorem of n-fold positive implicative filters is
obtained from the following result:
\begin{thm}\label{b1}Let $n\geq 1$.
Let $F_{1}$ and $F_{2}$ two filters of $L$  such that
$F_{1}\subseteq F_{2}$. If   $F_{1}$ is an n-fold positive
implicative  filter, then so is $F_{2}$.
\end{thm}
\begin{proof}
If   $F_{1}$ is an n-fold positive implicative filter, then by
Corollary \ref{lien11}, we get $\overline{x^{n}}\vee x\in F_{1}$ for
all $x\in L$. Since $F_{1}\subseteq F_{2}$, we have
$\overline{x^{n}}\vee x\in F_{2}$ for all $x\in L$ and  by Corollary
\ref{lien11}, $F_{2}$ is an n-fold positive implicative  filter.
\end{proof}
The following theorem  gives the relation between  n-fold positive
implicative filters and  n-fold  implicative filters in residuated
lattices.
\begin{thm}\label{lien1}
Every n-fold positive implicative filter of $L$ is an n-fold
implicative filter of $L$.
\end{thm}

\begin{proof}
Let $F$ be an n-fold positive implicative filter of $L$.
 Let $x, y\in L$ be such that  $x^{n+1}\longrightarrow y\in F$.
 By Prop. \ref{pro} we  have the
 following:
\begin{itemize}
\item[(1)]$(x^{n+1}\longrightarrow y)^{n}\longrightarrow (x^{n}\longrightarrow y)$=\\
$[(x^{n+1}\longrightarrow y)^{n-1}\otimes (x^{n+1}\longrightarrow
y)]\longrightarrow (x^{n}\longrightarrow
y)$=\\$[(x^{n+1}\longrightarrow y)^{n-1}]\longrightarrow
[(x^{n+1}\longrightarrow y)\longrightarrow (x^{n}\longrightarrow
y)$=\\$[(x^{n+1}\longrightarrow y)^{n-1}]\longrightarrow
[(x^{n+1}\longrightarrow y)\longrightarrow[
x^{n-1}\longrightarrow(x\longrightarrow
y)]$=\\$[(x^{n+1}\longrightarrow y)^{n-1}]\longrightarrow
[x^{n-1}\longrightarrow[(x^{n+1}\longrightarrow y)\longrightarrow
(x\longrightarrow y)]]$=\\$[(x^{n+1}\longrightarrow
y)^{n-1}]\longrightarrow
[x^{n-1}\longrightarrow[(x\longrightarrow(x^{n}\longrightarrow
y))\longrightarrow (x\longrightarrow y)]]$
\item[(2)]So by (1) we have :$(x^{n+1}\longrightarrow y)^{n}\longrightarrow (x^{n}\longrightarrow
y)$=\\$[(x^{n+1}\longrightarrow y)^{n-1}]\longrightarrow
[x^{n-1}\longrightarrow[(x\longrightarrow(x^{n}\longrightarrow
y))\longrightarrow (x\longrightarrow y)]]$
\item[ (3)] We have : $(x^{n}\longrightarrow
y)\longrightarrow y\leq [(x\longrightarrow(x^{n}\longrightarrow
y))\longrightarrow (x\longrightarrow y)]$
\item[ (4)] By (3) we have :$[(x^{n+1}\longrightarrow y)^{n-1}]\longrightarrow
[x^{n-1}\longrightarrow[(x^{n}\longrightarrow y)\longrightarrow
y]]\leq [(x^{n+1}\longrightarrow y)^{n-1}]\longrightarrow
[x^{n-1}\longrightarrow[[(x\longrightarrow(x^{n}\longrightarrow
y))\longrightarrow (x\longrightarrow y)]]]$
\item[ (5)] By (4) and (2) , we have : $((x^{n+1}\longrightarrow y)^{n-1})\longrightarrow
(x^{n-1}\longrightarrow((x^{n}\longrightarrow y)\longrightarrow
y))\leq (x^{n+1}\longrightarrow y)^{n}\longrightarrow
(x^{n}\longrightarrow y)$.
\item[ (6)]We have: $[(x^{n+1}\longrightarrow y)^{n-1}]\longrightarrow
[x^{n-1}\longrightarrow[(x^{n}\longrightarrow y)\longrightarrow
y]]= [(x^{n+1}\longrightarrow y)^{n-1}]\longrightarrow
[(x^{n}\longrightarrow y)\longrightarrow[x^{n-1}\longrightarrow
y]]= [(x^{n}\longrightarrow y)]\longrightarrow
[(x^{n+1}\longrightarrow
y)^{n-1}\longrightarrow[x^{n-1}\longrightarrow y]] $

\item[ (7)]By (6) and  (5) we get : $(x^{n}\longrightarrow y)\longrightarrow
[(x^{n+1}\longrightarrow
y)^{n-1}\longrightarrow(x^{n-1}\longrightarrow y)]\leq
(x^{n+1}\longrightarrow y)^{n}\longrightarrow
(x^{n}\longrightarrow y) $
\item[ (8)]We have : $(x^{n}\longrightarrow y)\otimes(x^{n+1}\longrightarrow
y)^{n-1}\otimes x^{n-1}= (x^{n}\longrightarrow
y)\otimes(x^{n+1}\longrightarrow
y)^{n-2}\otimes(x^{n+1}\longrightarrow y)\otimes x^{n-2}\otimes x=
(x^{n}\longrightarrow y)\otimes(x^{n+1}\longrightarrow
y)^{n-2}\otimes x^{n-2}\otimes x\otimes(x^{n+1}\longrightarrow y)
$
\item[ (9)]We also have : $x\otimes(x^{n+1}\longrightarrow y)=x\otimes[x\longrightarrow
(x^{n}\longrightarrow y)]\leq x^{n}\longrightarrow y$
\item[ (10)]Then : $(x^{n}\longrightarrow y)\otimes(x^{n+1}\longrightarrow
y)^{n-2}\otimes x^{n-2}\otimes x\otimes(x^{n+1}\longrightarrow
y)\leq (x^{n}\longrightarrow y)\otimes(x^{n+1}\longrightarrow
y)^{n-2}\otimes x^{n-2}\otimes (x^{n}\longrightarrow y)$
\item[ (11)]So: $(x^{n}\longrightarrow y)\otimes(x^{n+1}\longrightarrow
y)^{n-1}\otimes x^{n-1}\leq (x^{n}\longrightarrow
y)^{2}\otimes(x^{n+1}\longrightarrow y)^{n-2}\otimes x^{n-2}$
\item[ (12)]By (11)  we get: $[(x^{n}\longrightarrow
y)^{2}\otimes(x^{n+1}\longrightarrow y)^{n-2}\otimes
x^{n-2}]\longrightarrow y \leq [(x^{n}\longrightarrow
y)\otimes(x^{n+1}\longrightarrow y)^{n-1}\otimes x^{n-1}]
\longrightarrow y$
\item[ (13)]By (12),  we have: $((x^{n}\longrightarrow
y)^{2}\otimes(x^{n+1}\longrightarrow y)^{n-2})\longrightarrow
(x^{n-2}\longrightarrow y )\leq ((x^{n}\longrightarrow
y)\otimes(x^{n+1}\longrightarrow y)^{n-1})\longrightarrow (x^{n-1}
\longrightarrow y)$
\item[ (14)]So :  $(x^{n}\longrightarrow
y)^{2}\longrightarrow((x^{n+1}\longrightarrow
y)^{n-2}\longrightarrow (x^{n-2}\longrightarrow y ))\leq
(x^{n}\longrightarrow y)\longrightarrow((x^{n+1}\longrightarrow
y)^{n-1}\longrightarrow (x^{n-1} \longrightarrow y))$
\item[ (15)]By  (14) and (7), we have : $(x^{n}\longrightarrow
y)^{2}\longrightarrow((x^{n+1}\longrightarrow
y)^{n-2}\longrightarrow (x^{n-2}\longrightarrow y ))\leq
(x^{n+1}\longrightarrow y)^{n}\longrightarrow
(x^{n}\longrightarrow y) $
\end{itemize}
By repeating  (15) n times, we obtain: $(x^{n}\longrightarrow
y)^{n}\longrightarrow((x^{n+1}\longrightarrow
y)^{0}\longrightarrow (x^{0}\longrightarrow y ))\leq
(x^{n+1}\longrightarrow y)^{n}\longrightarrow
(x^{n}\longrightarrow y)$. This implies  $(x^{n}\longrightarrow
y)^{n}\longrightarrow(1\longrightarrow (1\longrightarrow y ))\leq
(x^{n+1}\longrightarrow y)^{n}\longrightarrow
(x^{n}\longrightarrow y)$.  Hence $(x^{n}\longrightarrow
y)^{n}\longrightarrow y \leq (x^{n+1}\longrightarrow
y)^{n}\longrightarrow (x^{n}\longrightarrow y)$.  Then
$((x^{n}\longrightarrow y)^{n}\longrightarrow y) \longrightarrow
((x^{n+1}\longrightarrow y)^{n}\longrightarrow
(x^{n}\longrightarrow y))= 1$.  Hence by  Prop. \ref{pro} we  have
:  $(x^{n+1}\longrightarrow y)^{n} \longrightarrow
(((x^{n}\longrightarrow y)^{n}\longrightarrow y) \longrightarrow
(x^{n}\longrightarrow y))= 1 \in F$.\\Since
$(x^{n+1}\longrightarrow y)\in F$  and $F$ is a filter, we have
$(x^{n+1}\longrightarrow y)^{n}\in F$. By the fact that $F$ is a
filter and $(x^{n+1}\longrightarrow y)^{n} \longrightarrow
(((x^{n}\longrightarrow y)^{n}\longrightarrow y) \longrightarrow
(x^{n}\longrightarrow y))\in F$, we have $((x^{n}\longrightarrow
y)^{n}\longrightarrow y) \longrightarrow (x^{n}\longrightarrow y)
\in F$. Since $F$ is an n-fold positive implicative filter, by
Prop. \ref{propo} we  have:  $x^{n}\longrightarrow y\in F$.\\ By
Prop. \ref{ch1}, $F$ is an n-fold  implicative filter.\\
\end{proof}

 The following Example shows that n-fold  implicative filters may not be n-fold positive
implicative filters.
\begin{ex}\label{ex2}
Let  $L$ be a residuated lattice from Example \ref{expp} and
$n\geq 1$. $F_{1}= \{1,a\}, F_{2}= \{1,b\}$ are n-fold implicative
filters but  not  n-fold positive implicative filters since  $
1\longrightarrow ((b^{n}\longrightarrow 0)\longrightarrow b)\in
F_{1}$
 and $ 1\in F_{1}$, but $ b\notin F_{1}$;  $ 1\longrightarrow ((a^{n}\longrightarrow
0)\longrightarrow a)\in F_{2}$
 and $ 1\in F_{2}$, but $ a\notin F_{2}$.
\end{ex}
\begin{prop}\label{propo15}
Every  n-fold positive implicative filter is an (n+1)-fold positive
implicative filter.
\end{prop}
\begin{proof}
Let   $F$  be an n-fold positive implicative filter of $L$.  Let
$x\in L$ such that $\overline{x^{n+1}}\longrightarrow x\in F$. We
show that $x\in F$.  Since $\overline{x^{n+1}}\longrightarrow x \leq
\overline{x^{n}}\longrightarrow x $, by the fact that  $F$ is a
filter, we get  $\overline{x^{n}}\longrightarrow x\in F$. Since $F$
is an n-fold positive implicative filter of $L$, by Prop.
\ref{propo},  we obtain $x\in F$. By Prop. \ref{propo}, $F$ is an
(n+1)-fold positive implicative filter of $L$.
\end{proof}
By the following example, we show that the converse of Prop.
\ref{propo15} is not true in general.

\begin{ex}
Let $L$ be a residuated lattice from Example \ref{exim3}.It is clear
that $\{1\}$ is an 3-fold positive implicative filter  but $\{1\}$
is not an 2-fold positive implicative filter, since
$\overline{b^{2}}\longrightarrow b \in \{1\}$ and $b\notin \{1\}$.
\end{ex}
\begin{defi} A
residuated lattice $L$ is called n-fold positive implicative
residuated lattice if it satisfies $ \overline{y^{n}}\longrightarrow
y=y$ for each $y\in L$.
\end{defi}
\begin{rem}
Since $ \overline{x^{n}}\longrightarrow x\leq
\overline{x}\longrightarrow x $ for all $x\in L $, it is clear that
1-fold residuated lattices are n-fold residuated lattices. It is
also clear that n-fold residuated lattices are (n+1)-fold residuated
lattices since $ \overline{x^{n+1}}\longrightarrow x\leq
\overline{x^{n}}\longrightarrow x $ for all $x\in L $
\end{rem}
The following example shows  that residuated lattices are not in
general n-fold positive implicative residuated lattices.
\begin{ex} \label{expp3}
Let  $L$ be a residuated lattice from Example \ref{expp}  and
$n\geq 2$.   $ L$ is not an  n-fold positive implicative residuated
lattice since $ (b^{n}\longrightarrow 0)\longrightarrow b=
(b\longrightarrow 0)\longrightarrow b= 0\longrightarrow b =1\neq b$
\end{ex}

 It is follows from Prop.\ref{propo} and  Prop.\ref{propo1} the following proposition:
\begin{prop}\label{propoo}
If $L$ is an n-fold positive implicative residuated lattice then the
notion of n-fold positive implicative filter and filter coincide.
\end{prop}
\begin{prop}\label{prch}
 The following conditions are
equivalent :
\begin{itemize}
\item[(i)]$L$ is an n-fold positive implicative  residuated lattice.
\item[(ii)] Every filter $F$ of $L$ is an n-fold positive implicative
filter of $L$
\item[(iii)] $\{1\}$ is an n-fold positive implicative filter
\end{itemize}
\end{prop}
\begin{proof}
$(i)\longrightarrow (ii)$ : Follows from Prop.\ref{propoo}\\
$(ii)\longrightarrow (iii)$ : Follows from the fact that $\{1\}$ is
a filter of $L$.\\
 $(iii)\longrightarrow (i)$: Assume that $\{1\}$ is
an n-fold positive implicative filter. Let $x\in L$.  By Corollary
\ref{lien11}, for all $x\in L$ holds $x\vee\overline{x^{n}}=1 $. By
(14)Prop. \ref{pro}, $x\vee\overline{x^{n}} \leq
(\overline{x^{n}}\longrightarrow x)\longrightarrow x$, Hence
$\overline{x^{n}}\longrightarrow x \longrightarrow x=1 $ or
equivalently $\overline{x^{n}}\longrightarrow x \leq x $, and by the
fact that $ x \leq\overline{x^{n}}\longrightarrow x $, we have
$\overline{x^{n}}\longrightarrow x = x $.
\end{proof}

The following result which follows from Prop. \ref{prch} and Prop.
\ref{ch5},  gives the relation between n-fold positive implicative
residuated lattices and  n-fold implicative residuated lattices.

\begin{prop}\label{cor15}
n-fold positive implicative residuated lattices are n-fold
implicative residuated lattices.
\end{prop}

By the following example, we show that the converse of Prop.
\ref{cor15} is not true in general.
\begin{ex}
A  residuated lattice $L$ from Example \ref{expp} is an n-fold
implicative  residuated lattice but by Example \ref{expp3}, it is
not an fold positive implicative residuated lattice.
\end{ex}

 The following result which follows from Prop. \ref{prch} and  Corollary \ref{lien11}, gives
new a characterization of n-fold positive implicative residuated
lattices.
\begin{cor}\label{prch34}   A
residuated lattice $L$ is an n-fold positive implicative
residuated lattice if and only if it satisfies $
\overline{y^{n}}\vee y=1$ for each $y\in L$.
\end{cor}

\begin{prop} \label{crch0}   If $L$ is a totaly ordered  residuated
lattice, then any n-fold positive implicative filter $F$ is
maximal filter of  $L$ and  $L/F$ is a locally finite residuated
lattice.
\end{prop}
\begin{proof}Let $L$ be a totaly ordered residuated
lattice. Assume that $F$ is n-fold positive implicative filter and
let $x\in L$ be such an element that $x\notin F$. From Prop.
\ref{propo} an assumption $\overline{x^{n}}\leq x$, or
equivalently $\overline{x^{n}}\ \longrightarrow x=1\in F$ leads to
a contradiction $x\in F$, so we necessarily have $x<
\overline{x^{n}}$. Therefore $x^{n+1}=0\in F$ and so
$\overline{x^{n+1}}=1\in F$ . The thesis follows from Prop.
\ref{max}.
\end{proof}
At consequence of Prop. \ref{crch0}, we have the following
results:

\begin{cor}\label{crch}
A totaly ordered residuated lattice is a locally finite  if
$\{1\}$ is an n-fold positive implicative filter.  A totaly
ordered   n-fold positive implicative  residuated lattice is a
locally finite.
\end{cor}
\begin{prop}\label{prcha}
A filter $F$ of $L$  is an  n-fold positive implicative filter if
and only $L/F$ is an n-fold positive implicative residuated
lattice.
\end{prop}
\begin{proof}
 Suppose that $F$  is an  n-fold positive implicative filter. Let  $x\in
 L$ be such that  $\overline{(x/F)^{n}}\longrightarrow x/F\in
\{1/F\}$, then  $(\overline{x^{n}}\longrightarrow x)/F=
\overline{(x/F)^{n}}\longrightarrow x/F=1/F$. So
$(\overline{x^{n}}\longrightarrow x)\in F$, by the fact that $F$
is an  n-fold positive implicative filter, we have  $x\in F$.
Hence $x/F\in \{1/F\}$, by Prop. \ref{propo}, $\{1/F\}$ is an
n-fold positive implicative filter of $L/F$, by Prop. \ref{prch},
$L/F$ is an n-fold positive implicative residuated lattice.\\
Suppose conversely that $L/F$ is an n-fold positive implicative
residuated lattice. Let   $x\in L$ be such that
$\overline{x^{n}}\longrightarrow x\in F$.  We have
$(\overline{x^{n}}\longrightarrow x)/F=1/F$, this implies
$\overline{(x/F)^{n}}\longrightarrow x/F\in \{1/F\}$. Since $L/F$
is an n-fold positive implicative residuated lattice, by Prop.
\ref{prch},  $\{1/F\}$ is an n-fold positive implicative filter of
$L/F$. Hence $x/F\in \{1/F\}$ or equivalently  $x\in F$.  By Prop.
\ref{propo},  $F$  is an  n-fold positive implicative filter of
$L$.
\end{proof}

The following examples shows that the notion of  n-fold positive
implicative residuated lattices exist.
\begin{ex}
Let  $L$ be a residuated lattice from Example \ref{expp} and
$n\geq 2$. Since $ F_{3}= \{1,a,b,c\}$ is an  n-fold positive
implicative filter, by Prop. \ref{prcha},  $L/F_{3}$ is an n-fold
positive implicative residuated lattice.
\end{ex}
Follows from Corollary \ref{lien11}, we have the following
proposition.
\begin{prop}\label{lien12}In any residuated lattices, the concepts
of  n-fold boolean filters and   n-fold positive implicative
filters coincide.
\end{prop}
\begin{defi}
$L$  is an n-fold boolean residuated lattice  if it  satisfies the
following conditions:\\
 $x\vee \overline{x^{n}} =1$ for all $x\in L$.
In particular 1-fold boolean residuated lattices are boolean
algebra.
\end{defi}
It is easy to observe that:
\begin{rem}\label{lien16}
A residuated lattice $L$ is an n-fold boolean residuated lattice
if and only if $\{1\}$  is an n-fold boolean filter of $L$ if and
only if $L$  is an n-fold  positive implicative residuated
lattice.
\end{rem}

\section{N-Fold Normal Filter in Residuated Lattices}
In \cite{B10}, the authors introduce the notion of n-fold normal
filter in BL-algebra. This motives us to introduce the notion of
n-fold normal filter in residuated lattices.
\begin{prop}\label{prochr}Let $n\geq 1$.
 The following conditions are
equivalent for any filter $F$:
\begin{itemize}
\item[(i)]For all  $x, y, z \in L$,
 if $ z \longrightarrow ((y^{n}\longrightarrow x)\longrightarrow x)\in F$
 and $ z\in F$, then $ (x\longrightarrow y)\longrightarrow y\in F$.
\item[(ii)]For all  $x, y \in L$, $(y^{n}\longrightarrow x)\longrightarrow x\in
F$ implies $(x\longrightarrow y)\longrightarrow y\in F$.
\end{itemize}
\end{prop}
\begin{proof}
Let $F$ be a filter which satisfying (i) . Assume that
$(y^{n}\longrightarrow x)\longrightarrow x \in F$. We have
$(y^{n}\longrightarrow x)\longrightarrow x = 1 \longrightarrow
((y^{n}\longrightarrow x)\longrightarrow x)\in F$ and $1\in F$, by
the fact that $F$ satisfying (i), we obtain $(x\longrightarrow
y)\longrightarrow y\in F$. Conversely, let  $ z \longrightarrow
((y^{n}\longrightarrow x)\longrightarrow x)\in F$
 and $ z\in F$, Since  $F$ is a filter, we have $(y^{n}\longrightarrow x)\longrightarrow x \in
 F$. By hypothesis, we obtain $(x\longrightarrow y)\longrightarrow y\in
 F$.
\end{proof}
\begin{defi}A filter
$F$  is an n-fold normal filter  if it satisfies one of the
conditions of Prop. \ref{prochr}.   In particular 1-fold normal
filters are normal  filters.
\end{defi}

\begin{ex}\label{de7}Let $n\geq 1$.
Let $L$ be a residuated lattice from Example \ref{exp}. Simple
computations proves that $ F_{3}= \{1,a,b,c\}$ is an  n-fold normal
filter.
\end{ex}
The following example shows that filters may not be n-fold normal in
general.
\begin{ex}\label{de8}
Let $L$ be a residuated lattice from Example \ref{exp}.  $ F_{2}=
\{1,b\}$ is not an  n-fold normal filter since
$(a^{n}\longrightarrow 0)\longrightarrow 0=1\in \{1,b\}$ but $
[(0\longrightarrow a)\longrightarrow a]= a\notin\{1,b\}$.
\end{ex}
\begin{prop}\label{propo4}
  n-fold positive implicative filters are n-fold normal filters.
\end{prop}
\begin{proof}
 Assume that $F$ is an n-fold positive implicative filter an let
$(x^{n}\longrightarrow y)\longrightarrow y \in F$.\\ We must  show
that $(y\longrightarrow x)\longrightarrow x \in F$. \\  Since
 $y\leq (y\longrightarrow x)\longrightarrow x$, by (4)
 Prop.\ref{pro}, we obtain $(x^{n}\longrightarrow y)\longrightarrow y \leq
 (x^{n}\longrightarrow y)\longrightarrow ((y\longrightarrow x)\longrightarrow x )$.\\
 From this and the fact that $F$ is a filter, we obtain \
$ (x^{n}\longrightarrow y)\longrightarrow ((y\longrightarrow
x)\longrightarrow x)\in F$.\\ Since $x\leq (y\longrightarrow
x)\longrightarrow x$, by (4)Prop.\ref{pro} \ we have $x^{n}\leq
((y\longrightarrow x)\longrightarrow x)^{n}$, hence\\
$(x^{n}\longrightarrow y)\longrightarrow
  [(y\longrightarrow x)\longrightarrow x ]\leq
  ([(y\longrightarrow x)\longrightarrow x]^{n}\longrightarrow y)\longrightarrow
  [(y\longrightarrow x)\longrightarrow x ] $.\\ Since $ (x^{n}\longrightarrow y)\longrightarrow [(y\longrightarrow
x)\longrightarrow x]\in F$,
   by the fact that $F$ is a filter, we obtain \
    $([(y\longrightarrow x)\longrightarrow x]^{n}\longrightarrow y)\longrightarrow
  [(y\longrightarrow x)\longrightarrow x ]\in F$. \ From this and  the fact that $F$ is
   an n-fold positive implicative
  filter, we obtain the result by apply Prop.\ref{propo}.
\end{proof}
By the following example, we show that the converse of Prop.
\ref{propo4}  is not true in general.
\begin{ex}
Let $ L$ be a lattice from Example \ref{exim3}
 $\{1\}$ is an 1-fold normal filter but $\{1\}$ is not an 1-fold positive implicative
 filter.
\end{ex}
\section{n-fold fantastic Filter in Residuated Lattices}
\begin{defi}Let $n\geq 1$.
 $F$  is an n-fold fantastic filter $L$ if it satisfies the
following conditions:
\begin{itemize}
\item[(i)]$1\in F$
\item[(ii)]For all  $x, y \in L$,
$y\longrightarrow x\in F$ implies $ [(x^{n}\longrightarrow
y)\longrightarrow y]\longrightarrow x\in F$.
\end{itemize}
In particular 1-fold fantastic  filters are fantastic
filters.\cite{B6}
\end{defi}
\begin{ex}Let $n\geq 1$.
Let $L$ be a residuated lattice from Example \ref{expob2}. It is
easy to check that $\{1\}$ is an n-fold fantastic filter.
\end{ex}

The following example shows that filters may not be n-fold
fantastic in general.

\begin{ex}\label{exf3}
Let  $L$ be a residuated lattice from Example \ref{expob1}.
$\{1\}$ is not an n-fold fantastic filter since $a\longrightarrow
c=1\in \{1\}$ but $ [(c^{n}\longrightarrow a)\longrightarrow
a]\longrightarrow c= c\notin\{1\}$.
\end{ex}
\begin{prop}\label{lp}Let $n\geq 1$.
  n-fold positive
implicative filters  are  n-fold fantastic filters.
\end{prop}
\begin{proof}Assume that $F$ is an n-fold positive
implicative filter. Let  $x, y\in L$  be such that
$y\longrightarrow x\in F$.\\ By Prop. \ref{pro}, we have:\\
 $ x \leq [((x^{n}\longrightarrow y)\longrightarrow y)\longrightarrow x]$. \ (a) \\
Then by Prop. \ref{pro}, we also have:\\
$ x^{n} \leq [((x^{n}\longrightarrow y)\longrightarrow y)\longrightarrow x]^{n}$. \ (b)\\
By (b) and Prop. \ref{pro}, we get,  $ (x^{n}\longrightarrow y) \geq
[((x^{n}\longrightarrow y)\longrightarrow y)\longrightarrow x]^{n}\longrightarrow y$. \ (c)\\
By  Prop. \ref{pro}, we get, $ y\longrightarrow x\leq  ((x^{n}\longrightarrow y)\longrightarrow y)\longrightarrow((x^{n}\longrightarrow y)\longrightarrow x)$.\ (d)\\
We also have :\\ $((x^{n}\longrightarrow y)\longrightarrow y)\longrightarrow((x^{n}\longrightarrow y)\longrightarrow x)=((x^{n}\longrightarrow y)\longrightarrow(((x^{n}\longrightarrow y)\longrightarrow y)\longrightarrow x)$.\ (e)\\
So, by (d)and(e), we get,  $ y\longrightarrow x\leq ((x^{n}\longrightarrow y)\longrightarrow(((x^{n}\longrightarrow y)\longrightarrow y)\longrightarrow x)$.\ (f)\\
By (c) and Prop. \ref{pro}, we get,\\
 $[((x^{n}\longrightarrow y)\longrightarrow[((x^{n}\longrightarrow
y)\longrightarrow y] \longrightarrow x]\leq
[[((x^{n}\longrightarrow y)\longrightarrow y)\longrightarrow
x]^{n}\longrightarrow y]\longrightarrow [[((x^{n}\longrightarrow
y)\longrightarrow y] \longrightarrow x]]$    (g)\\
By (f) and (g) , we obtain,\\
$ y\longrightarrow x\leq [[((x^{n}\longrightarrow
y)\longrightarrow y)\longrightarrow x]^{n}\longrightarrow
y]\longrightarrow [[((x^{n}\longrightarrow
y)\longrightarrow y] \longrightarrow x]]$    (h)\\
Since $F$ is a filter (see Prop. \ref{propo1}), by (h) and the
fact that $y\longrightarrow x\in F$, we get :\\
$[[((x^{n}\longrightarrow y)\longrightarrow y)\longrightarrow
x]^{n}\longrightarrow y]\longrightarrow [[((x^{n}\longrightarrow
y)\longrightarrow y] \longrightarrow x]]\in F$.  (w)\\ By (w),
Prop. \ref{propo} and the fact that  $F$ is an n-fold positive
implicative filter, we obtain $[[((x^{n}\longrightarrow
y)\longrightarrow y] \longrightarrow x]]\in F$. \\
Hence $F$ is an n-fold fantastic filter.\\
\end{proof}

The following example shows that n-fold fantastic filters may not be
n-fold positive implicative in general.

\begin{ex}Let $n\geq 1$.
Let $L$ be a residuated lattice from Example \ref{expob2}. It is
easy to check that $\{1\}$ is an n-fold fantastic filter, but not
n-fold positive implicative  filter since
$\overline{a^{n}}\longrightarrow a \in F$ and $a\notin F$.
\end{ex}
\begin{prop}\label{lpp}Let $n\geq 1$.
  n-fold fantastic filters are  n-fold normal filters.
\end{prop}
\begin{proof}Assume that $F$ is an n-fold fantastic filter. Let  $x, y\in L$  be such that
$(x^{n}\longrightarrow y)\longrightarrow y \in F$ and $t=(y\longrightarrow x)\longrightarrow x $.  We must show that $t \in F$.  \\ By Prop. \ref{pro}, we have:\\
 $ y \leq (y\longrightarrow x)\longrightarrow x$,
   so  $(x^{n}\longrightarrow y)\longrightarrow y \leq [(x^{n}\longrightarrow y)\longrightarrow [(y\longrightarrow x)\longrightarrow x]]$,
  that is  $(x^{n}\longrightarrow y)\longrightarrow y \leq [(x^{n}\longrightarrow y)\longrightarrow t]$. \ (a) \\
  Since $(x^{n}\longrightarrow y)\longrightarrow y \in F$,
by (a) an the fact that $F$ is a filter, it follows that
$(x^{n}\longrightarrow y)\longrightarrow t\in F$. \ (b) \\
By (b) and the fact that  $F$ is an n-fold fantastic filter,  we
get that $[(t^{n}\longrightarrow(x^{n}\longrightarrow
y))\longrightarrow (x^{n}\longrightarrow y)]\longrightarrow t \in
F$. \ (c) \\
By Prop. \ref{pro},   we also have
$t^{n}\longrightarrow(x^{n}\longrightarrow y)=
x^{n}\longrightarrow(t^{n}\longrightarrow y)$,  so
$(t^{n}\longrightarrow(x^{n}\longrightarrow y))\longrightarrow
(x^{n}\longrightarrow y )=
(x^{n}\longrightarrow(t^{n}\longrightarrow y))\longrightarrow
(x^{n}\longrightarrow y )$ and then \\
$[(x^{n}\longrightarrow(t^{n}\longrightarrow y))\longrightarrow
(x^{n}\longrightarrow y)]\longrightarrow t \in
F$. \ (d) \\
On the other hand, by Prop. \ref{pro},   we also have
$(t^{n}\longrightarrow y)\longrightarrow y \leq
(x^{n}\longrightarrow(t^{n}\longrightarrow y))\longrightarrow
(x^{n}\longrightarrow y)$. \ (e) \\ Since $x\leq t$,  it follows
that  $(x^{n}\longrightarrow y)\longrightarrow y \leq
(t^{n}\longrightarrow y)\longrightarrow y $.   \ (f) \\  Since $F$
is a filter and  $(x^{n}\longrightarrow y)\longrightarrow y \in
F$,  by (f) we obtain  $(t^{n}\longrightarrow y)\longrightarrow y
\in F$. \ (g) \\ Since $F$ is a filter,  by (g) and (e) it follows
that  $ (x^{n}\longrightarrow(t^{n}\longrightarrow
y))\longrightarrow (x^{n}\longrightarrow y)  \in F$.  \ (h) \\
Since $F$ is a filter,  by (h) and (d) it follows that $t \in
F$.\\ Hence $F$ is an n-fold normal filter.
\end{proof}

\begin{lem}\label{l1}
For all  $x, y \in L$, we have:\\
$[(x^{n}\longrightarrow x^{2n})\otimes (x^{2n}\longrightarrow
y)]\leq x^{n}\longrightarrow y$.
\end{lem}
\begin{proof}Let  $x, y\in L$, by Prop.\ref{pro}  we have the
following:
\begin{itemize}
\item[(1)]$x^{n}\longrightarrow x^{2n}\leq [(x^{2n}\longrightarrow y)\longrightarrow (x^{n}\longrightarrow y)]$
\item[(2)] By  (1) we  have : $[(x^{2n}\longrightarrow y)\otimes(x^{n}\longrightarrow x^{2n})]\leq [(x^{2n}\longrightarrow y)\otimes((x^{2n}\longrightarrow y)\longrightarrow (x^{n}\longrightarrow y))]$
\item[(3)]$[(x^{2n}\longrightarrow y)\otimes((x^{2n}\longrightarrow
y)\longrightarrow (x^{n}\longrightarrow y))]\leq
x^{n}\longrightarrow y$
\end{itemize}
 By  (2) and (3) we  have :
$[(x^{n}\longrightarrow x^{2n})\otimes (x^{2n}\longrightarrow
y)]\leq x^{n}\longrightarrow y$.
\end{proof}

\begin{prop}\label{l3}Let $n\geq 1$.
Let $F$ a filter of $L$. If $F$  is   n-fold fantastic  and n-fold
implicative filter, then $F$  is an n-fold positive implicative
filter.
\end{prop}
\begin{proof}
Let  $x, y\in L$  be such that $(x^{n}\longrightarrow
y)\longrightarrow x\in F$. Assume that $F$  is both  n-fold
fantastic filter and n-fold implicative filter.\\ Since $F$  is a
n-fold fantastic filter, by the fact that $(x^{n}\longrightarrow
y)\longrightarrow x\in F$, we have : $ [(x^{n}\longrightarrow
(x^{n}\longrightarrow y))\longrightarrow (x^{n}\longrightarrow
y)]\longrightarrow x\in F$.\\
By Lemma \ref{l1} and residuation we get : $(x^{n}\longrightarrow
x^{2n})\leq  (x^{2n}\longrightarrow y)
\longrightarrow(x^{n}\longrightarrow y)$. So,
$(x^{n}\longrightarrow x^{2n})\leq [x^{n}\longrightarrow
(x^{n}\longrightarrow y)] \longrightarrow(x^{n}\longrightarrow
y)$. Hence,  $[(x^{n}\longrightarrow x^{2n})\longrightarrow x]\geq
[[x^{n}\longrightarrow (x^{n}\longrightarrow y)]
\longrightarrow(x^{n}\longrightarrow y)]\longrightarrow x$. Since
$F$ is a filter, by the fact that $ [(x^{n}\longrightarrow
(x^{n}\longrightarrow y))\longrightarrow (x^{n}\longrightarrow
y)]\longrightarrow x\in F$, we have :  $(x^{n}\longrightarrow
x^{2n})\longrightarrow x \in F$. Since $F$  is and n-fold
implicative filter, by  Prop. \ref{ch1}, we get, $x\in F$. By
Prop. \ref{propo},  $F$  is an n-fold positive implicative filter.

\end{proof}
Follows from Prop. \ref{lp}, Prop. \ref{l3}, Prop. \ref{lien1}, it
is easy to show the following theorem.
\begin{thm}\label{l4}Let $n\geq 1$.
Let $F$ a filter. $F$  is an n-fold positive implicative filter of
 $L$ if and only if   $F$ is  n-fold fantastic  and n-fold implicative
filter of $L$.
\end{thm}

\begin{defi}\label{defrch}
 $L$ is said to be   n-fold fantastic residuated lattice if for all  $x, y \in L$,
$y\longrightarrow x=  [(x^{n}\longrightarrow y)\longrightarrow
y]\longrightarrow x$.
\end{defi}
The following example shows that the notion of  n-fold fantastic
residuated lattice exist.
\begin{ex}\label{exf1}Let $n\geq 1$.
Let  $L$ be a residuated lattice from Example \ref{expob2}. It is
easy to check that  $L$ is an   n-fold fantastic residuated
lattice.
\end{ex}
The following example shows that residuated lattices may not be
n-fold fantastic in general.
\begin{ex}\label{exf2}
Let  $L$ be a residuated lattice from Example \ref{expob1}. $L$ is
not an n-fold fantastic residuated lattice since $a\longrightarrow
c=1\neq c=  [(c^{n}\longrightarrow a)\longrightarrow
a]\longrightarrow c$.
\end{ex}

The following proposition gives a characterization of n-fold
fantastic residuated lattice.
\begin{prop}\label{l2}
$L$ is an   n-fold fantastic residuated lattice if and only if the
inequality  $ (x^{n}\longrightarrow y)\longrightarrow y\leq
(y\longrightarrow x) \longrightarrow x$ holds for all $x, y \in
L$.
\end{prop}
\begin{proof}
Assume that $L$ is an n-fold fantastic  residuated lattice. Let
$x, y \in L$.  We have $ [(x^{n}\longrightarrow y)\longrightarrow
y]\longrightarrow [(y\longrightarrow x) \longrightarrow x]=
(y\longrightarrow x)\longrightarrow[[(x^{n}\longrightarrow
y)\longrightarrow y]\longrightarrow x]$.  (a)\\ By hypothesis $
y\longrightarrow x=[[(x^{n}\longrightarrow y)\longrightarrow
y]\longrightarrow x]$. Hence $ (y\longrightarrow
x)\longrightarrow[[(x^{n}\longrightarrow y)\longrightarrow
y]\longrightarrow x]=1$.  (b)\\ By (a) and  (b), we get $
[(x^{n}\longrightarrow y)\longrightarrow y]\longrightarrow
[(y\longrightarrow x) \longrightarrow x]=1$ or equivalently  $
[(x^{n}\longrightarrow y)\longrightarrow y]\leq
[(y\longrightarrow x) \longrightarrow x]$.\\
Suppose conversely that  the inequality  $ (x^{n}\longrightarrow
y)\longrightarrow y\leq (y\longrightarrow x) \longrightarrow x$
holds for all $x, y \in L$.  Then $ (y\longrightarrow
x)\longrightarrow[[(x^{n}\longrightarrow y)\longrightarrow
y]\longrightarrow x]=[(x^{n}\longrightarrow y)\longrightarrow
y]\longrightarrow [(y\longrightarrow x) \longrightarrow x] $.  (e)\\
Since $ (x^{n}\longrightarrow y)\longrightarrow y\leq
(y\longrightarrow x) \longrightarrow x$, by Prop. \ref{pro}, we
get  $[(x^{n}\longrightarrow y)\longrightarrow
y]\longrightarrow[(x^{n}\longrightarrow y)\longrightarrow y]
\leq[(x^{n}\longrightarrow y)\longrightarrow y]\longrightarrow
[(y\longrightarrow x) \longrightarrow x]$, that is $1
\leq[(x^{n}\longrightarrow y)\longrightarrow y]\longrightarrow
[(y\longrightarrow x) \longrightarrow x]$ or equivalently
$[(x^{n}\longrightarrow y)\longrightarrow y]\longrightarrow
[(y\longrightarrow x) \longrightarrow x]= 1$.  (f)\\  By (e) and
(f), its follows that $ (y\longrightarrow
x)\leq[[(x^{n}\longrightarrow y)\longrightarrow y]\longrightarrow
x]$.   (g)\\Since  $ y\leq (x^{n}\longrightarrow y)\longrightarrow
y$, we also get by Prop. \ref{pro}, $ y\longrightarrow
x\geq[[(x^{n}\longrightarrow y)\longrightarrow y]\longrightarrow
x]$.   (h)\\
From (h) and (g), we obtain $ y\longrightarrow x=
[[(x^{n}\longrightarrow y)\longrightarrow y]\longrightarrow x]$.
Hence  $L$ is an n-fold fantastic  residuated lattice.
\end{proof}
\begin{prop}\label{prch7}
 The following conditions are
equivalent for any filter $F$:
\begin{itemize}
\item[(i)]$L$ is an n-fold fantastic  residuated lattice.
\item[(ii)] Every filter $F$ of $L$ is an n-fold fantastic
filter of $L$
\item[(iii)] $\{1\}$ is an n-fold fantastic filter of $L$.
\end{itemize}
\end{prop}
\begin{proof}
$(i)\longrightarrow (ii)$ : Follows from Definition \ref{defrch} \\
$(ii)\longrightarrow (iii)$ : Follows from the fact that $\{1\}$
is a filter of $L$.\\
 $(iii)\longrightarrow (i)$: Assume that $\{1\}$ is
an n-fold fantastic filter.\\Let $x, y \in L$ and $t=
(y\longrightarrow x) \longrightarrow x$. By Prop. \ref{pro},
$y\leq t$. So $y\longrightarrow t =1$ and by the hypothesis, we
have
$[(t^{n}\longrightarrow y)\longrightarrow y]\longrightarrow t=1$,\\
that is $[(t^{n}\longrightarrow y)\longrightarrow y]\leq t$.  (w)\\
On the other hand, $x\leq t$ implies $x^{n}\leq t^{n}$, hence
$[(x^{n}\longrightarrow y)\longrightarrow y]\leq
(t^{n}\longrightarrow y)\longrightarrow y $. (z)\\
By (z) and (w), it follows that $[(x^{n}\longrightarrow
y)\longrightarrow y]\leq t=   (y\longrightarrow x) \longrightarrow
x$. Hence by Prop.  \ref{l2}, $L$ is an n-fold fantastic
residuated lattice.
\end{proof}
Combine  Prop.\ref{prch7}, Prop.\ref{prch}, Prop.\ref{ch5}
 and Theorem \ref{l4}, we have the following result:
\begin{cor}\label{l90}Let $n\geq 1$.
 $L$  is an n-fold positive implicative residuated lattice if and
only if   $L$ is  n-fold fantastic residuated lattice and n-fold
implicative residuated lattice.
\end{cor}
The following corollary gives a characterization of n-fold
fantastic filter in  residuated lattice.
\begin{cor}\label{cor1}Let $n\geq 1$.
Let $F$ be a filter of $L$.  Then   $F$ is an n-fold fantastic
filter if and only  $L/F$ is is an n-fold fantastic  residuated
lattice.
\end{cor}

\begin{proof}
Let $F$ be a filter of $L$. Assume that $F$ is an n-fold fantastic
filter. We show that  $L/F$ is is an n-fold fantastic residuated
lattice. Let $x,y \in L$ be such that $y/F \longrightarrow x/F \in
\{1/F\}$, then  $(y\longrightarrow x)/F =1/F$ or equivalently
$y\longrightarrow x\in F$. Since $F$ is an n-fold fantastic
filter, we get  $  [(x^{n}\longrightarrow y)\longrightarrow
y]\longrightarrow x \in F $ or equivalently  $
([(x^{n}\longrightarrow y)\longrightarrow y]\longrightarrow x )/F=
1/F $, so $([((x/F)^{n}\longrightarrow y/F)\longrightarrow
y/F]\longrightarrow x/F )\in  \{1/F\}$. Hence $\{1/F\}$ is an
n-fold fantastic filter of $L/F$, therefore  by Prop. \ref{prch7},
$L/F$ is is an n-fold fantastic  residuated lattice.\\ Conversely,
assume that $L/F$ is is an n-fold fantastic  residuated lattice.
Let $x,y \in L$ be such that $y\longrightarrow x \in F$ then
$(y\longrightarrow x)/F =1/F$ or equivalently $y/F\longrightarrow
x/F \in \{1/F\}$. Since $L/F$ is is an n-fold fantastic residuated
lattice,  by Prop. \ref{prch7}, $\{1/F\}$ is an n-fold fantastic
filter of $L/F$. From this and the fact that $y/F\longrightarrow
x/F \in \{1/F\}$, we have: $([((x/F)^{n}\longrightarrow
y/F)\longrightarrow y/F]\longrightarrow x/F )\in  \{1/F\}$ or
equivalently  $ ([(x^{n}\longrightarrow y)\longrightarrow
y]\longrightarrow x )/F= 1/F $, so $  [(x^{n}\longrightarrow
y)\longrightarrow y]\longrightarrow x \in F $. Hence  $F$ is an
n-fold fantastic filter.
\end{proof}
The extension theorem of  n-fold fantastic   filters is obtained
from the following result:
\begin{thm}\label{f2}Let $n\geq 1$.
Let $F_{1}$ and $F_{2}$ two filters of $L$  such that
$F_{1}\subseteq F_{2}$.  If   $F_{1}$ is an n-fold fantastic
filter, then so is $F_{2}$.
\end{thm}
\begin{proof}

Let $x,y\in L$ be such that $y\longrightarrow x\in F_{2}$. Since
$F_{1}$ is an n-fold fantastic  filter, by Corollary. \ref{cor1},
$L/F_{1}$ is an n-fold positive implicative residuated lattice. So
$([((x/F_{1})^{n}\longrightarrow y/F_{1})\longrightarrow
y/F_{1}]\longrightarrow x/F_{1} )= y/F_{1}\longrightarrow x/F_{1}
$, so $(y\longrightarrow x)\longrightarrow([(x^{n}\longrightarrow
y)\longrightarrow y]\longrightarrow x )\in F_{1} $, so
$(y\longrightarrow x)\longrightarrow([(x^{n}\longrightarrow
y)\longrightarrow y]\longrightarrow x )\in F_{2} $. Since  $F_{2}$
is a filter of $L$, by the fact that $y\longrightarrow x\in F_{2}$
and $(y\longrightarrow x)\longrightarrow([(x^{n}\longrightarrow
y)\longrightarrow y]\longrightarrow x )\in F_{2} $, we get
$([(x^{n}\longrightarrow y)\longrightarrow y]\longrightarrow x
)\in F_{2} $.  Hence  $F_{2}$ is an n-fold fantastic filter.
\end{proof}

\section{n-fold obstinate Filters in Residuated Lattices}

\begin{defi}\label{defob1}
A filter  $F$  is an n-fold obstinate filter of $L$  if it
satisfies the following conditions for any $n\geq 1$:
\begin{itemize}
\item[(i)]$0\notin F$
\item[(ii)]For all  $x, y \in L$,
$x,y\notin F$ imply $ x^{n}\longrightarrow y\in F$ and  $
y^{n}\longrightarrow x \in F$
\end{itemize}
In particular 1-fold obstinate  filters are obstinate
filters.\cite{B14}
\end{defi}

The following proposition gives a characterization of  n-fold
obstinate filter  of $ L$.

\begin{prop}\label{propob1}For any $n\geq 1$,
 the following conditions are
equivalent for any filter $F$:
\begin{itemize}
\item[(i)]$F$ is an n-fold obstinate filter
\item[(ii)]For all  $x \in L$, if $x\notin F$ then     there exist $m\geq
1$ such that   $(\overline{x^{n}})^{m}\in F$
\end{itemize}
\end{prop}
\begin{proof}
$(i)\longrightarrow (ii)$  :    Suppose that  $F$ is n-fold
obstinate filter of $L$. Let $x \in L$ be such that $ x\notin F$.
By setting $y=0$ in Definition \ref{defob1}, we get: $
x^{n}\longrightarrow 0\in F$. Hence for $m=1$, we have:
$(\overline{x^{n}})^{m}\in F$\\
$(ii)\longrightarrow (i)$: Conversely,  let $x,y\notin F$, we show
that  $ x^{n}\longrightarrow y\in F$ and  $ y^{n}\longrightarrow x
\in F$. By the hypothesis, there are $ m_{1}, m_{2}\geq 1$ such
that $(\overline{x^{n}})^{m_{1}},(\overline{y^{n}})^{m_{2}} \in
F$.\\ By Prop. \ref{pro}(8), we have:\\
$(\overline{x^{n}})^{m_{1}}\leq \overline{x^{n}}\leq
x^{n}\longrightarrow y $   \ (a) \\ and
$(\overline{y^{n}})^{m_{2}}\leq \overline{x}^{n}\leq
y^{n}\longrightarrow x$   \ (b) \\
Since  $F$ is a filter, by (a) and (b), we get   $
x^{n}\longrightarrow y\in F$ and  $ y^{n}\longrightarrow x \in F$.
\end{proof}
\begin{ex} \label{expb1}
Let $ L $ be a lattice from Example \ref{expob1}.
 $F= \{1,b,c,d\}$ is a proper filter  of $L$. Using Prop.
 \ref{propob1}, for any $n\geq 1$,
it is easy to check that $F$  is an n-fold obstinate filter of
$L$.
\end{ex}

The following example shows that any filters may  not be n-fold
obstinate filter.
\begin{ex} \label{expb2}
Let $ L $ be a lattice   from  Example \ref{expob2}.
 $F= \{1,c,d\}$ is a proper filter  of $L$.  For any $n\geq 1$, we
 have: $a,b\notin F$ but
$a^{n}\longrightarrow b=b\notin F$. Hence $F$  is not an n-fold
obstinate filter of $L$.
\end{ex}
\begin{prop}
Every n-fold obstinate filter of  $L$ is a (n+1)-fold obstinate
filter of $L$.
\end{prop}
\begin{proof}
Let  $F$  be  an n-fold obstinate filter of $L$ and $x,y \notin
F$. We  show that  $ x^{n+1}\longrightarrow y\in F$ and  $
y^{n+1}\longrightarrow x \in F$.   By  hypothesis, $
x^{n}\longrightarrow y\in F$ and  $ y^{n}\longrightarrow x \in F$.
 (c) \\ By Prop. \ref{pro}(8), $ x^{n+1}\leq x^{n}$ and $
y^{n+1}\longrightarrow y^{n}$.\\ Then, $ x^{n}\longrightarrow
y\leq x^{n+1}\longrightarrow y$ and  $ y^{n}\longrightarrow x\leq
y^{n+1}\longrightarrow x $. \ (d)  \\
By (c) and (d) and the fact that $F$ is a filter, we obtain  $
x^{n+1}\longrightarrow y\in F$ and  $ y^{n+1}\longrightarrow x \in
F$. Hence $F$   is a (n+1)-fold obstinate filter of $L$.
\end{proof}

The extension theorem of n-fold obstinate filters is obtained from
the following result   and any $n\geq 1$:
\begin{thm}\label{th2}
Let $F_{1}, F_{2} $ two filter of  $L$ be such that
$F_{1}\subseteq F_{2} $. If  $F_{1} $ is an n-fold obstinate
filter of $L$ then so is $F_{2} $.
\end{thm}

\begin{proof}
Let $F_{1}, F_{2} $ two filter of  $L$ be such that
$F_{1}\subseteq F_{2} $. Assume that  $F_{1} $ is an n-fold
obstinate filter of $L$, and let   $x\notin F_{2} $. Since
$F_{1}\subseteq F_{2} $, we have $x\notin F_{1} $. Since  $F_{1} $
is an n-fold obstinate filter of $L$, by Prop. \ref {propob1},
there exist $m\geq 1$ such that   $(\overline{x^{n}})^{m}\in
F_{1}$. Since $F_{1}\subseteq F_{2} $, we have
$(\overline{x^{n}})^{m}\in F_{2}$. Therefore there exist $m\geq 1$
such that   $(\overline{x^{n}})^{m}\in F_{2}$. Hence by Prop. \ref
{propob1},  $F_{2} $ is an n-fold obstinate filter of $L$.
\end{proof}
From Theorem \ref{th2}, it is easy to shows the following result
for any $n\geq 1$:
\begin{cor}\label{cor3}
$\{1\} $ is an n-fold obstinate filter of  $L$ if and only if
every filter of  $L$ is an  n-fold obstinate filter of  $L$.
\end{cor}

\begin{prop}\label{propob3}
 The following conditions are
equivalent for any filter $F$  and any $n\geq 1$:
\begin{itemize}
\item[(i)]$F$ is an n-fold obstinate filter
\item[(ii)]$F$ is a maximal and n-fold positive implicative filter
\item[(iii)]$F$ is a maximal and n-fold  implicative filter
\end{itemize}
\end{prop}
\begin{proof}
$(i) \longrightarrow (ii)$:  Assume that  $F$ is an n-fold
obstinate filter.   We show that $F$ is a maximal  and n-fold
positive implicative filter.  At first   we show that $F$ is a
maximal.  Let $x\notin F $, since $F$ is an n-fold obstinate
filter, by Prop. \ref{propob1}, there exist $m\geq 1$ such that
$(\overline{x^{n}})^{m}\in F$.  Since $(\overline{x^{n}})^{m}\leq
\overline{x^{n}} $, by the fact that $F$ is a filter, we get
$\overline{x^{n}}\in F$, by Prop.
\ref{max}, $F$ is a maximal filter.\\
On the other hand, assume in the contrary that there exist $x\in L
$ such that $\overline{x^{n}}\longrightarrow x\in F$ and $x\notin
F$. \ Since $F$ is an n-fold obstinate filter, by Prop.
\ref{propob1}, there exist $m\geq 1$ such that
$(\overline{x^{n}})^{m}\in F$.  Since $(\overline{x^{n}})^{m}\leq
\overline{x^{n}} $, by the fact that $F$ is a filter, we get
$\overline{x^{n}}\in F$. Since $F$ is a filter,  by the fact that
$\overline{x^{n}}\longrightarrow x\in F$, we obtain $x\in F$ which
is a contradiction. Hence  for all  $x\in L $,
$\overline{x^{n}}\longrightarrow x\in F$ implies $x\in F$.\\ By
Prop. \ref{propo}, $F$ is an n-fold positive implicative filter.\\
$(ii) \longrightarrow (iii)$: follows from Prop. \ref{lien1}\\
$(iii) \longrightarrow (i)$:  Assume that   $F$ is a maximal and
n-fold implicative filter of $L$.  Let $x, y L $ be such that $x,
y \notin F$. By Lemma \ref{lemob}, $L_{x}=\{b\in L:
x^{n}\longrightarrow b\in F \}$ is a filter of $L$. $L_{y}=\{b\in
L: y^{n}\longrightarrow b\in F \}$ is a filter of $L$.\\ Let $z\in
F$, since $z\leq x^{n}\longrightarrow z$, we have $
x^{n}\longrightarrow z \in F  $, hence $ x^{n}\longrightarrow z\in
F  $, so    $z\in L_{x}$ and we obtain $F\subseteq L_{x}$. On the
other hand, $x^{n}\longrightarrow x = 1 \in F$ since
$x^{n}\subseteq x $, hence  $x\in L_{x}$. By hypothesis, $x \notin
F$, So $F\varsubsetneq L_{x} \subseteq L$. Since  $F$ is a maximal
filter of $L$, we get $ L_{x} = L$.  Therefore $y\in L_{x}$ or
equivalently $ x^{n}\longrightarrow y\in F  $. Similarly, we get $
y^{n}\longrightarrow x\in F$.  Hence $F$ is an n-fold obstinate
filter of  $ L$.
\end{proof}
\begin{prop}\label{propob4}
 The following conditions are
equivalent for any filter $F$ and any $n\geq 1$:
\begin{itemize}
\item[(i)]$F$ is an n-fold obstinate filter
\item[(ii)]$F$ is a maximal and n-fold boolean filter
\item[(iii)]$F$ is a prime of second kind and n-fold  boolean filter
\end{itemize}
\end{prop}
\begin{proof}
$(i) \longrightarrow (ii)$:  Assume that  $F$ is an n-fold
obstinate filter.  First observe that by Prop. \ref{propob3}, $F$
is a maximal filter. Let $x\in L$. We have two cases: $x\in F$ or $x\notin F$   \\
case 1: $x\in F$. Since  $x\leq x\vee \overline{x^{n}}$, by the
fact that $F$ is a filter, we  have  $ x\vee \overline{x^{n}}\in
F$.\\
case 2: $x\notin F$.  Since $F$ is an n-fold obstinate filter, by
Prop.  \ref{propob1},  there exist $m\geq 1$ such that
$(\overline{x^{n}})^{m}\in F$. Since $(\overline{x^{n}})^{m}\leq
\overline{x^{n}}\leq  x\vee \overline{x^{n}} $, we have
$(\overline{x^{n}})^{m}\leq x\vee \overline{x^{n}}$.
   By the fact that $F$ is a filter, we  have  $ x\vee
\overline{x^{n}}\in F$.\\
Since in both the two cases  $x\vee \overline{x^{n}}\in F$, it is
clear that for all $x\in L$, $x\vee \overline{x^{n}}\in F$, hence
$F$ is an n-fold  boolean filter.\\
$(ii) \longrightarrow (iii)$: Follows in the fact that a maximal
filter of $L$ is a  prime filter of second kind of  $L$.\\
 $(iii) \longrightarrow(i)$: Assume that $F$ is a prime filter of the second kind and n-fold  boolean
 filter.   Let $x\in L$  be such that $x\notin F$.    $F$ is  an n-fold  boolean
 filter, we have $x   \vee \overline{x^{n}}\in F$. Since  $F$ is a
 prime filter of the second kind  , by the fact $x\notin F$, we have $ \overline{x^{n}}^{1}\in F$.
 by Prop.  \ref{propob1},  $F$ is an n-fold obstinate filter.
\end{proof}

Combine Prop. \ref{propob3} and  Prop. \ref{lp}, we have the
following proposition.

\begin{prop}\label{propob5}
 Any n-fold obstinate  filter $F$ is an  n-fold fantastic filter. The
 converse is not true in general.
\end{prop}
The following example shows that the
 converse of the Proposition \ref{propob5}  is not true in general.

\begin{ex}
Let $L$ be a residuated lattice from Example \ref{expob2}. It is
easy to check that $\{1\}$ is an n-fold fantastic filter but not
n-fold obstinate  filter.
\end{ex}
Follows from Prop.\ref{propob4} and Prop.\ref{lien12}, we have the
following result:
\begin{cor}\label{propob31}
 The following conditions are
equivalent for any filter $F$  and any $n\geq 1$:
\begin{itemize}
\item[(i)]$F$ is an n-fold obstinate filter
\item[(ii)]$F$ is a prime filter in the second kind and n-fold positive implicative filter
\end{itemize}
\end{cor}
\begin{prop}\label{propob6}for any $n\geq 1$,
 a filter $F$ is an n-fold obstinate  filter  if and only if every filter of
 $L/F$ is an n-fold obstinate  filter of  $L/F$.
\end{prop}
\begin{proof}
 Assume that   $F$ is an n-fold obstinate  filter.  Let  $x,y \in
 L$ be such that $x/F,y/F \notin \{1/F\}$, then  $x,y \notin F$. Since  $F$ is an n-fold obstinate
 filter, we have $x^{n}\longrightarrow y,y^{n}\longrightarrow x \notin
 F$. From this we have,  $(x^{n}\longrightarrow y)/F, (y^{n}\longrightarrow x)/F\notin
 1/F$ or equivalently $(x/F)^{n}\longrightarrow y/F,(y/F)^{n}\longrightarrow x/F \notin
 1/F$.  Hence  $\{1/F\}$ is an n-fold obstinate  filter of  $L/F$
 and by Corollary \label{cor3}, every filter of
 $L/F$ is an n-fold obstinate  filter of  $L/F$.\\
 Conversely,  let $x,y \notin F$. Then $x/F,y/F \notin \{1/F\}$.
 Since $\{1/F\}$ is an n-fold obstinate  filter of  $L/F$, we have
 $(x/F)^{n}\longrightarrow y/F,(y/F)^{n}\longrightarrow x/F \notin
 1/F$ or equivalently $(x^{n}\longrightarrow y)/F, (y^{n}\longrightarrow x)/F\notin
 1/F$. So  $x^{n}\longrightarrow y,y^{n}\longrightarrow x \notin
 F$ and  $F$ is an n-fold obstinate  filter.
\end{proof}

\begin{defi}
A residuated lattice  $L$ is said to be  an n-fold obstinate
residuated it satisfies the following condition :\\  For all
$x,y\in L$,   $x,y\neq 1$ implies $x^{n}\longrightarrow y=1$ and
$y^{n}\longrightarrow x=1$.
\end{defi}
\begin{prop}\label{propob18}
 The following conditions are
equivalent for any $n\geq 1$:
\begin{itemize}
\item[(i)]$L$ is an n-fold obstinate  residuated lattice
\item[(ii)]$\{1\}$ is an  n-fold obstinate filter of $L$.
\item[(iii)]Every filter of $L$ is n-fold obstinate
\end{itemize}
\end{prop}
\begin{proof}
$(i)\longrightarrow (ii)$: Obvious \\
$(ii)\longrightarrow (iii)$:Follows from Theorem \ref{th2} \\
$(iii)\longrightarrow (i)$:Assume that every filter of $L$ is
n-fold obstinate, then $\{1\}$ is an  n-fold obstinate filter of
$L$ since $\{1\}$ is a filter of $L$. The thesis follows by
setting $F= \{1\}$ in  Definition \ref{defob1}.
\end{proof}

The following example shows that the notion of n-fold obstinate
residuated lattice exist.
\begin{ex}\label{ex97}
Let $ L $ be a lattice from Example \ref{expob1}.
 $F$  is an n-fold obstinate filter of $L$, then  by Prop. \ref{propob6} and Prop. \ref{propob18},
  $L/F$ is an n-fold obstinate
residuated lattice, for any $n\geq 1$.
\end{ex}

The following example shows that  any residuated lattice may not
be n-fold obstinate residuated lattice.
\begin{ex}\label{expp9}
Let $ L $ be a lattice from Example \ref{exp}. For any $n\geq 1$,
 $F=\{1, d\}$  is not an n-fold obstinate filter of $L$, then by Prop. \ref{propob6} and Prop. \ref{propob18},  $L$ is not an n-fold obstinate
residuated lattice.
\end{ex}
Follows from Prop. \ref{propob18}, Prop. \ref{propob4} and Prop.
\ref{lien16}, we have the following proposition.
\begin{prop}\label{propob54}
n-fold obstinate  residuated lattices are n-fold boolean
residuated lattices
\end{prop}

\section{ Diagram among  type of n-fold filters in Residuated Lattices}

\end{document}